\documentclass[11pt]{article}

\title{On local weak limit and subgraph counts for sparse random graphs\footnote{Part of this work was supported by the Research Council of Lithuania (MIP-067/2013).} }

\usepackage{amsmath}
\usepackage{amssymb}
\usepackage{latexsym}
\usepackage{color}
\usepackage{enumerate}
\usepackage{hyperref}
\hypersetup{
    colorlinks=false,
    pdfborder={0 0 0},
}

\usepackage[nottoc,numbib]{tocbibind}

\newenvironment{proof}{\noindent{\bf Proof}}{\hspace*{\fill}$\Box$}
\newenvironment{proofof}[1]{%
\noindent {\bf Proof of #1}}%
{\hspace*{\fill}$\Box$}

\newtheorem{theorem}{Theorem}[section]
\newtheorem{lemma} {Lemma}[section]
\newtheorem{corollary} {Corollary}[section]
\newtheorem{prop} {Proposition}[section]

\newtheorem{remark} {Remark}[section]

\newtheorem{claim} {Claim}[section]
\newtheorem{example} {Example}[section]

\def\E{{\mathbb E}\,}

\let\eps=\epsilon
\def\enddiscard{}
\long\def\discard#1\enddiscard{}

\def\Po{\mbox{\rm Po}}
\newcommand{\pr}{\mathbb P}

\newcommand{\ca}{{\mathcal A}}

\newcommand{\cf}{{\mathcal F}}

\newcommand{\cR}{{\mathcal R}}
\newcommand{\cl}{{\mathcal L}}
\newcommand{\ct}{{\mathcal T}}

\newcommand{\cv}{{\mathcal V}}
\newcommand{\cx}{{\mathcal X}}
\newcommand{\ii}{{\mathbb I}}
\newcommand{\bp}{{\mathbf p}}
\newcommand{\Xn}[1]{{X^{(#1)}}}

\newcommand{\Vn}[1]{{V^{#1}}}

\newcommand{\toD}{{\xrightarrow{d}}}

\newcommand{\dn}{d_n}

\newcommand{\weaklyto}{{\,\toD\,}}
\newcommand{\weaklypto}{{\,\xrightarrow{p}\,}}

\usepackage{amsopn}
\DeclareMathOperator{\emb}{emb}
\DeclareMathOperator{\rr}{root}

\date{2021-09-17}

\begin{document}

\author{{Valentas Kurauskas}
\\
\\ Vilnius University}


\maketitle

\begin{abstract}
    We use an inequality of Sidorenko to 
    show a general relation between local and global subgraph counts and degree moments
    for locally weakly convergent sequences of sparse random graphs.
    This yields an optimal criterion to check when the asymptotic behaviour of
    graph statistics such as the clustering coefficient and assortativity
    is determined by the local weak limit.

    As an application we obtain new facts 
    for several common models of
    sparse random intersection graphs
    where the local weak limit, as we see here,
    is a simple random clique tree corresponding to a certain two-type Galton-Watson branching process.

    \smallskip

\emph{keywords}: local weak limit, subgraph count, clique tree, random intersection graph

\smallskip


\end{abstract}

\bigskip

\small {
  \bigskip

  Author's address: Akademijos 4, LT-08663 Vilnius, Lithuania. Email: valentas@gmail.com.
  }

 \newpage

 \section{Introduction}
 \label{sec.intro0}

 A \emph{rooted graph} is a pair $(H, v)$
 where $H$ is a graph and $v \in V(H)$ is a distinguished vertex called the \emph{root}.
 We often use only the symbol $H$ to denote $(H,v)$; in this case we write $\rr(H)=v$.
 For a graph $G$ and its vertex $v$, let $B_r$ be the function that maps $(G,v)$ to the
 the rooted graph $(H, v)$ where $H$ is
 the subgraph induced on the vertices of $G$ with distance from $v$ at most $r$. 
 We simplify $B_r(G,v) = B_r((G,v))$ for $B_r$ and other functions on rooted graphs.

A graph is \emph{locally finite} if the degree of each of its vertices is finite. Denote by $\cong$
the isomorphism relation between connected rooted graphs
which preserves the root.
Let ($\mathcal{G}_*$, $d_{loc}$) be the space of rooted connected locally finite graphs with
equivalence relation $\cong$ and distance
\[
    d_{loc}(G_1, G_2)  = 2^{-\sup \{r: \, B_r(G_1) \cong B_r(G_2)\}}.
\] 
Consider a sequence of finite graphs $\{G_n, n=1,2, \dots\}$.
In this paper we assume $|V(G_n)| \ge 1$ for $n \ge 1$.
Let $v^*_n$ be a uniformly random vertex  from $V(G_n)$.
The component of $G_n$ containing $v^*_n$ together with root $v^*_n$
induces a Borel measure $\mu_n$ on  $(\mathcal{G}_*, d_{loc})$ for each $n$. 
Let $\mu^*$ be another Borel measure on $(\mathcal{G}_*, d_{loc})$,
and denote by $G^*$ a random element\footnote{Without loss of generality we assume that all random objects
we define in the paper are random elements in a single probability space $(\Omega, \cf, \pr)$
with the specified laws; the integration $\E$ is over $\Omega$.} with law $\mu^*$.
Following Benjamini and Schramm \cite{benjaminischramm2001}, Aldous, Lyons and Steele and other authors \cite{aldoussteele2004, lovasz-hom-book, lyons}, 
we say that $G^*$ is the \emph{local weak limit} of $\{(G_n,v^*_n)\}$ and write $(G_n,v^*_n) \weaklyto G^*$ 
if and only if the measures $\mu_n$ converge weakly
to $\mu^*$: for each continuous bounded function $f: (\mathcal{G}_*, d_{loc}) \to \mathbb{R}$
\begin{equation}\label{eq.measureconv}
   \E f(G_n, v^*_n) \to \E f(G^*).
\end{equation}
Here and below all limits are as $n\to \infty$, unless stated otherwise. 
Since $(\mathcal{G}_*, d_{loc})$ is separable and complete \cite{aldouslyons}, a standard argument (e.g., Theorem~2.3 of \cite{billingsley_weak}) shows that $(G_n, v^*_n) \weaklyto G^*$ if and only if for each non-negative integer $r$ and 
each rooted connected graph $H$ 
\[
    \pr(B_r(G_n, v^*_n) \cong H) \to \pr(B_r(G^*) \cong H).
\]
We focus on models of random graphs with bounded average degree. 
Among others,  
the inhomogeneous random graph
model of Bollob\'{a}s, Janson and Riordan \cite{bollobasjansonriordan2011} and
the preferential attachment model, see Berger, Borgs, Chayes and Saberi \cite{bergerborgschayessaberi2014}
have been shown to have a weak limit (in an explicit form).
Recently such a limit was also shown to exist for random planar graphs \cite{stufler2019}.
The local weak limit, if it exists, yields a lot of information about the asymptotics of various graph parameters, see, e.g., 
\cite{benjaminilyonsschramm2013, bollobasjansonriordan2011, bollobasriordan2011, bordenavelelarge2009, lovasz-hom-book, salez2011}. 

The present contribution consists of a general result, Theorem~\ref{thm.gencounts}, relating the asymptotics of subgraph counts 
with the local weak limit, and its application in the area of random intersection graphs.

The structure of the paper is as follows. In Section~\ref{sec.gencounts} we present and prove Theorem~\ref{thm.gencounts}. In a separate result, Theorem~\ref{thm.riglocal} of Section~\ref{sec.intro}, we determine the (very simple) local weak limit of several popular random intersection graph models.
Combining this with Theorem~\ref{thm.gencounts} and using the fact that many important graph parameters can be expressed in terms of small subgraph counts, we obtain a number of previous and some new results for this type of models, 
see Section~\ref{sec.applications}. The same method works for any sparse random graph model where we have weak local convergence (see, e.g., Section~\ref{subsec.app.gen}). 

The first manuscript of this paper was completed and posted to arXiv in 2015 \cite{kurauskas2015}. 
Since then there has appeared some work in a similar general direction, including, for example, \cite{vadon2019}, unaware of the very general Theorem~\ref{thm.gencounts}. A recent book in preparation \cite{vanderhofstad2020} also devotes a chapter to weak limits as a general technique to study real world networks.
The present version of the paper fixes several minor errors and omissions and has an updated literature list. 

\section{Local weak limit and subgraph counts}
\label{sec.gencounts}

In Section 7 of \cite{bollobasriordan2011} Bollob\'as, Janson and Riordan remark that the local weak limit
does not always determine the global subgraph count asymptotics, see also Example~\ref{ex.1} below. 
They propose an extra condition of ``exponentially bounded tree counts''. 
Our main result is
that a simple condition on the degree moment is sufficient and, in general, necessary.


A homomorphism from a graph $H$ to a graph $G$ is a mapping from $V(H)$ to $V(G)$ that
maps adjacent vertices in $H$ to adjacent vertices in $G$.
Denote by $\emb(H, G)$ the number of embeddings (injective homomorphisms)
from $H$ to $G$. 
For a rooted graph $H'$ 
let $\emb'(H', G, v)$ denote the number of
embeddings from $H'$ to $G$ that map $\rr(H')$ to $v$.   Let $\cR(H)$ denote the set of all $|V(H)|$ possible rooted graphs obtained from a graph $H$. Finally let $d_G(v)$ denote the degree of vertex $v$ in $G$.

\begin{theorem}\label{thm.gencounts}
    Let $h\ge2$ be an integer,
    let $\{G_n, n = 1, 2\dots\}$ be a sequence graphs,
    such that $n_1 = n_1(n) = |V(G_n)| \to \infty$ and $n_1 \ge 1$, 
    let $v^*_n$ be chosen uniformly at random from $V(G_n)$
    and suppose $(G_n,v^*_n) \weaklyto G^*$. Write 
    $\dn = d_{G_n}(v^*_n)$,
    $d^* = d_{G^*}(r^*)$,
    where $r^* = \rr(G^*)$, 
    and assume $\E (d^*)^{h-1} < \infty$. 
    Then the following statements are equivalent:
    \begin{enumerate}[(i)]
        \item $\E \dn^{h-1} \to \E (d^*)^{h-1}$;  \label{eq.degmoment}
        \item $\dn^{h-1}$ is uniformly integrable; \label{eq.uint}
        \item for any connected graph $H$ on $h$ vertices and any $H' \in \cR(H)$  \label{eq.embconv}
            \[
                n_1^{-1} \emb(H, G_n) \to \E \emb'(H', G^*, r^*).
            \]
    \end{enumerate}
\end{theorem}

The above theorem provides a sufficient condition for
the continuous but not necessarily bounded
function $f_H :(\mathcal{G}_*, d_{loc}) \to \mathbb{R}$ defined by $f_H(G, v)= \emb'(H, G, v)$
to satisfy (\ref{eq.measureconv}).
It is easy to construct weakly convergent sequences for 
which (\ref{eq.degmoment})-(\ref{eq.embconv}) fail to hold: 
\begin{example}\label{ex.1}
    Let $(G_n,v^*_n)$ be as in Theorem~\ref{thm.gencounts} and assume $|V(G_n)|= n$.
    Let $(G'_n, v^*_n)$ be obtained by merging edges of a clique on a subset $S_n$ of $G_n$. 
If $|S_n| = \Omega(n^{1/h})$ and $|S_n| = o(n)$ 
    then $(G_n',v^*_n) \weaklyto G^*$,
but (\ref{eq.degmoment})--(\ref{eq.embconv}) do not hold for $G_n'$. 
\end{example}
The proof of this theorem follows from the next basic but not widely known result of
Sidorenko \cite{sidorenko1994}.
\begin{theorem}{(Sidorenko, 1994)}\label{lem.sidorenko}
    Let $H$ be a connected graph on $h$ vertices. Then for any graph $G$
    \[
        \hom(H, G) \le \hom(K_{1,h-1}, G) = \sum_{v \in V(G)} d_G(v)^{h-1}.
    \]
\end{theorem}
Here $\hom(H, G)$ is the number of homomorphisms from $H$ to $G$ and $K_{1, s}$ is the complete bipartite graph
with part sizes $1$ and $s$.
A special case where $H$ is a path has been rediscovered in \cite{fiolgarriga2012}, see also \cite{csikvarilin2013}.


Below, in Section~\ref{sec.applications}, we restate Theorem~\ref{thm.gencounts} in a random setting and demonstrate how it implies general results on network statistics expressible through subgraph counts such as the clustering and assortativity coefficients.


Recall that a sequence of random variables $\{X_n, n=1,2,\dots\}$ is uniformly integrable
if $\sup_{a\to \infty} \sup_n \E |X_n| \ii_{|X_n| > a} = 0$, equivalently,
if $\E |X_n| \ii_{|X_n| > \omega_n} \to 0$ for any $\omega_n \to \infty$.
A basic fact, see e.g. \cite{billingsley_1999} p. 31--32, is
\begin{lemma}\label{lem.uint}
    Suppose random variables $X^*, X_n, n=1,2,\dots$ are non-negative, integrable and $X_n$ converges to $X^*$ in distribution as $n \to \infty$.
    Then $\{X_n\}$ is uniformly integrable if and only if $\E X_n \to \E X^*$.
\end{lemma}

    \medskip

\begin{proofof}{Theorem~\ref{thm.gencounts}} In the proof denote $G=G_n$ and $v^*=v^*_n$.

\emph{(\ref{eq.degmoment})$\Leftrightarrow$(\ref{eq.uint})}.
    $\dn$ converges in distribution
    to $d^*$ by (\ref{eq.measureconv}). 
Thus $\dn^{h-1}$ converges in distribution to $(d^*)^{h-1}$ and 
the proof follows by Lemma~\ref{lem.uint}. 
    
    \emph{(\ref{eq.degmoment})$\Rightarrow$(\ref{eq.embconv})}.
    Suppose (\ref{eq.degmoment}) holds. Fix any connected graph $H$ with $|V(H)| = h$.
    Let $r$ be the diameter of $H$. Write $b_j(G,v) = |B_j(G, v)|$
    and $b_j^* = |B_j(G^*)|$. Note that $\pr(b_j^* = \infty) = 0$ for $j=0,1,\dots$ since $G^*$ is locally finite. 
    For a rooted graph $H'$ denote 
    \[
        X(H') = \emb'(H', G, v^*) \quad \mbox{and} \quad X^*(H') = \emb'(H', G^*, r^*).
    \]
    Using Sidorenko's theorem,
    Theorem~\ref{lem.sidorenko}, for any $H' \in \cR(H)$
    \begin{align}
        &\E X(H') = n_1^{-1} \emb(H,G) \le
         n_1^{-1} \hom(H,G) \nonumber
       \\& \le  n_1^{-1} \hom(K_{1,h-1},G)
       = \E \dn^{h-1} \to  \E (d^*)^{h-1} < \infty. \label{eq.subdeg}
   \end{align}
   Next, we have $X(H') \to X^*(H)$ in distribution for each $H' \in \cR(H)$ (apply
   (\ref{eq.measureconv}) to the continuous and bounded function $f(G,v) = \ii_{emb'(G, H', v) = k}$).
   Therefore by Fatou's lemma and (\ref{eq.subdeg}),
   \begin{equation} \label{eq.fatou}
       \E X^*(H') \le \lim \inf \E X(H') < \infty.
    \end{equation}
    Let $\eps \in (0,1)$. Since $\E (d^*)^{h-1}$ and $\E X^*(H')$ 
    are finite, we can find a $t > 0$ such that
    \begin{align*}
        &\E (d^*)^{h-1} \ii_{d^* > t} \le \E (d^*)^{h-1} \ii_{b^*_{r+1} > t} < \eps \quad \mbox{ and }
       \\ &\E X^*(H') \ii_{b_{r+1}^* > t} 
    < \eps \mbox{ for each } H' \in \cR(H).
    \end{align*}
    Pick $s \ge t $ large enough that 
        $\pr(b_{r+1}^* > s) \le 0.5 \epsilon t^{-(h + r - 1)}$.
    By Lemma~\ref{lem.uint} and  (\ref{eq.measureconv})  for each $H' \in \cR(H)$
    \begin{align} 
        &\E \dn^{h-1} \ii_{\dn \le t} \to \E (d^*)^{h-1} \ii_{d^* \le t}; \label{eq.dk}
     \\   &\E X(H') \ii_{b_{r+1}(G, v^*) \le s} \to \E X^*(H') \ii_{b_{r+1}^* \le s} \ge \E X^*(H') -\eps; \label{eq.copies_lower}
    \\ &\E X(H') \ii_{b_{r+1}(G, v^*) \in (t,s]} \to \E X^*(H') \ii_{b_{r+1}^* \in (t,s]} \le \eps; \label{eq.copies_tail}
    \\   &\pr(b_{r+1}(G, v^*) > s) \to \pr(b_{r+1}^* > s)  \le 0.5 \epsilon t^{-(h+r-1)}.  \label{eq.R2bound}
    \end{align} 
     Define subsets of $V(G)$:
    \[
        R_1 := \{v: d_G(v) > t\}; \quad R_2 := \{v: b_{r+1}(G, v) > s\}. 
    \]
    We call an embedding $\sigma$ of $H$ into $G$ \emph{bad} if its image shares a vertex with $R_1 \cup R_2$. Denote the set of all bad embeddings by $\cx_{bad}$. Note that for $H' \in \cR(H)$ 
    \begin{equation} \label{eq.cxbad}
        0 \le \emb(H, G) - n_1 \E X(H') \ii_{b_{r+1}(G, v^*) \le s} \le |\cx_{bad}|.
   \end{equation}
    Let $\cx_1$ bet the set of all embeddings $\sigma$ whose image intersects both $R_1$
    and $V \setminus (R_1 \cup R_2)$. Let $\cx_2 = \cx_{bad} \setminus \cx_1$.
    By Theorem~\ref{lem.sidorenko}, the number of bad embeddings which have the image entirely contained in $R_1$ is  
    \begin{align}
        \emb(H, G[R_1]) &\le \hom(H, G[R_1]) \le \sum_{v \in R_1} d_G(v)^{h-1} \nonumber
         =
         n_1 (\E \dn^{h-1} - \E \dn^{h-1}\ii_{\dn \le t}) 
      \\ &\le n_1 \E \dn ^{h-1} -  n_1\E (d^*)^{h-1} + \epsilon n_1 + o(n_1)  \le \epsilon n_1 + o(n_1).  \label{eq.dk3}
    \end{align}
    Here the last two inequalities follow by  (\ref{eq.degmoment}) and (\ref{eq.dk}). 
    Let $v \in V \setminus (R_1 \cup R_2)$ be a vertex
    in the image of an embedding in $\cx_1$.
By the definition of $R_1$ and $R_2$, $b_{r+1}(G, v) \in (t, s]$. So using (\ref{eq.copies_tail})
    \[
    |\cx_1| \le  n_1 \sum_{H' \in \cR(H)} \E X(H') \ii_{b_{r+1}(G,v^*) \in (t, s]} \le h \eps n_1 + o(n_1).
    \]  
    Now consider a subgraph $H_\sigma$ of $G$, $H_\sigma \cong H$ corresponding to an embedding $\sigma \in \cx_2$.
    $H_\sigma$ cannot have an edge in 
    $E_1 = \{xy \in G: x \in R_1, y \in V(G) \setminus (R_1 \cup R_2)\}$,
    otherwise $\sigma$ would be an element of $\cx_1 = \cx_{bad}\setminus \cx_2$.
    So $V(H_\sigma)$ is contained in $R_1 \cup Q$, where
    \[
        Q = \bigcup_{v \in R_2} V(B_r(G - E_1, v)) \setminus R_1.
    \]
    Note that since each vertex in $Q$ has degree at most $t$, $|Q| \le 2 |R_2| t^r$.
    By Theorem~\ref{lem.sidorenko}
    \begin{equation}\label{eq.cx2}
        |\cx_2| \le \sum_{v \in R_1} d_G(v)^{h-1} + \sum_{v \in Q} d_G(v)^{h-1}.
    \end{equation}
    For the second term we have by (\ref{eq.R2bound}) 
    \begin{equation} \label{eq.cx22}
      \sum_{v \in Q} d_G(v)^{h-1} \le |Q| t^{h-1} \le 2 |R_2| t^r t^{h-1} \le \epsilon n_1 + o(n_1).
   \end{equation}
    Combining (\ref{eq.dk3}), (\ref{eq.cx2}) and (\ref{eq.cx22}) we obtain
    $|\cx_2| \le 2 \epsilon n_1 + o(n_1)$.
    We have proved
    \[
      |\cx_{bad}| = |\cx_1| + |\cx_2| \le (h + 2) \epsilon n_1 + o(n_1).
    \]
    Since the proof holds for arbitrarily small $\epsilon$, 
    we see that $n_1^{-1} |\cx_{bad}| \to 0$. Thus (\ref{eq.embconv}) follows using 
    (\ref{eq.copies_lower})
    and (\ref{eq.cxbad}).


    \emph{(\ref{eq.embconv})$\Rightarrow$(\ref{eq.degmoment}).}
    Write $(x)_p = x (x-1) \dots (x-p+1)$.
    (\ref{eq.embconv}) applied to $H = K_{1, h-1}$ yields
    $\E (\dn)_{h-1} \to \E (d^*)_{h-1}$, while $(G,v^*) \weaklyto G^*$
    shows that $(\dn)_{h-1} \to (d^*)_{h-1}$ in distribution.
    Thus $(\dn)_{h-1}$ is uniformly integrable by Lemma~\ref{lem.uint}. 
    This implies that for
    any $j = 1, 2, \dots, h-1$ 
    $(\dn)_j \le (\dn)_{h-1}$ is uniformly integrable, so by Lemma~\ref{lem.uint} again $\E (\dn)_j \to \E (d^*)_j$. 
    Using $S(h-1,j)$ to denote Stirling numbers of the second kind,
    \begin{align*}
        & \E (\dn)^{h-1} = \sum_{j=1}^{h-1} S(h-1,j) \E (\dn)_j
        \to \sum_{j=1}^{h-1} S(h-1,j) \E (d^*)_j = \E (d^*)^{h-1}. 
    \end{align*}

  \end{proofof}


  The next fact is simple and known (cf. Lemma~9.3 of \cite{bollobasjansonriordan2011}),
  but we include a proof for completeness.
  \begin{lemma}\label{lem.insensitive} Suppose $(G_n,v^*_n) \weaklyto G^*$ and the degree $d_n$ of
      a uniformly random vertex $v_n^*$ from $V(G_n)$ is uniformly integrable. Write $n_1 = n_1(n) = |V(G_n)|$.
      Let $G_n'$ be obtained from $G_n$ by adding or removing edges incident to
      a set $S_n \subseteq V(G_n)$ of size $o(n_1)$. Then $(G_n',v^*_n) \weaklyto G^*$.
  \end{lemma}

  \begin{proofof}{Example~\ref{ex.1}}
      It is straightforward that the uniform integrability condition (\ref{eq.uint}) fails for $G_n'$,
      so the other two conditions also fail by Theorem~\ref{thm.gencounts}. The
      fact that $G_n$ and $G_n'$ have the same local weak limit $G^*$ follows from
      Lemma~\ref{lem.insensitive}. 
  \end{proofof}

  \bigskip

  \begin{proofof}{Lemma~\ref{lem.insensitive}} 
      Denote by $N_n$ the set of vertices in $V(G_n) \setminus  S_n$
  which have a neighbour in $S_n$.
  
  \begin{claim} \label{claim.neighbourhood}
          For any $\epsilon \in (0,1)$ there are $\delta > 0$, $n_0 > 0$ such
          that if $n \ge n_0$ and $0 < |S_n| < \delta n_1$ then $|N_n| \le \epsilon n_1$.
  \end{claim}

  \begin{proof}\,
      Let $\delta$ and $n_0$ be such that $\delta < \epsilon$, $\E d_n \ii_{d_n > 0.5 \epsilon \delta^{-1}} < 0.5 \epsilon$
      for all $n \ge n_0$. Assume that  $|N_n| > \epsilon n_1$ for some $n \ge n_0$. Write $d(v) = d_{G_n}(v)$.
      We have
      \begin{align*}
          &n_1 \E d_n \ii_{d_n > 0.5 \epsilon \delta^{-1}} \ge \sum_{v \in S_n} d(v) \ii_{d(v) > 0.5 \epsilon \delta^{-1}} 
          \ge \sum_{v \in S_n} (d(v) - 0.5 \epsilon \delta^{-1}) \ii_{d(v) > 0.5 \epsilon \delta^{-1}} 
          \\ & \ge |S_n| \left(|S_n|^{-1} \sum_{v \in S_n} d(v) - 0.5 \epsilon \delta^{-1} \right)
          \ge  \epsilon n_1 - 0.5 \epsilon n_1 \ge 0.5 \epsilon n_1,
      \end{align*}
      which is a contradiction. Here we used Jensen's inequality
      and the assumption $\sum_{v\in S} d(v) \ge |N_n| > \epsilon n_1$.
  \end{proof}

  \medskip

  Now fix any positive integer $r$. As is done in \cite{bollobasjansonriordan2011},
  we apply the above claim $r$ times to get that the set $N_n^{(r)}$ of vertices
  at distance at most $r$ from $S_n$ in $G_n$ has size $o(n_1)$. Now $(G_n', v_n^*) \weaklyto G^*$
  follows since
  \[
      \pr(B_r(G_n, v^*_n) \cong B_r(G_n', v^*_n)) \ge \pr (v^*_n \not \in N_n^{(r)}) = 1 - o(1).
  \]
\end{proofof}

\section{Uncorrelated random clique trees}
\label{sec.intro}

Random intersection graphs were introduced in \cite{kssc99} and received some attention as a 
potential model for large empirical networks with clustering; see, e.g., surveys \cite{bgjkr2014a, bgjkr2014b}.
We show that in the regime which yields sparse graphs with a positive clustering
coefficient in such models the weak limit is very specific, namely it is  
an uncorrelated random clique tree, defined formally below.

Let $H = (V^1, V^2, E)$ be a bipartite graph.
The \emph{intersection graph} $G = G(H)$ of $H$ is the graph on the vertex set $V(G) = V^1$ with edges 
\[E(G) = \{uv: \exists w\in V^2 \mbox { such that } uw, wv \in H\},\]
where $e \in H$ is a shorthand for $e \in E(H)$.
An intersection graph of a random bipartite graph $H$ is called a \emph{random intersection graph}.  
It will be convenient to assume that $V^i$ consists
of the first $n_i$ elements of a countable set ${\mathcal V}^i$, where
$\cv^1 \cap \cv^2 = \emptyset$. 
The set $V^2$ is often called the set of \emph{attributes}\footnote{The names $V$ and $W$ are often used
in the literature for $V^1$ and $V^2$.}. We will call
the elements of ${\mathcal V}^i$ \emph{vertices of type $i$}.
For $v \in V^i$ we denote $S_v = \Gamma_H(v)$, $X_v = |S_v|$ where $\Gamma_H(x)$ is the set of neighbours of $v$ in the graph $H$.
Sometimes we will want to stress the type of $v$ in the notation. Since $V^i = \{v_1^{(i)}, v_2^{(i)} \dots\}$ consists
of the first $n_i$ vertices of $\cv_i$, for $v = v_j^{(i)}$ we will set $X_v^{(i)} := X_v$ and $S_v^{(i)} := S_v$.
We will denote by $X \sim Y$ the fact that $X$ and $Y$
have the same distribution.

Many different variants of the random bipartite graph $H$ have been studied, see, e.g., survey papers \cite{bgjkr2014a, bgjkr2014b}.

The \emph{active} random intersection graph: each $v \in V^1$ independently chooses $X^{(1)}_v$ from 
        a distribution $P$ on $\{0,\dots,n_2\}$, then draws a uniformly random subset $S^1_v$ of size $X^{(1)}_v$ of its neighbours from $V^2$ (independently of other vertices). A special case is the \emph{binomial} random intersection graph.

The \emph{passive} random intersection graph: each $v \in V^2$ independently chooses $X^{(2)}_v$ from 
        a distribution $P$ on $\{0,\dots,n_1\}$, then draws a uniformly random subset $S^2_v$ of size $X^{(2)}_v$ of its neighbours from $V^1$ (independently of other vertices).
 
The \emph{inhomogeneous} random intersection graph $G^{inhomog}(n_1, n_2, \xi^{(1)}, \xi^{(2)})$: the vertices $v \in V^i$ are independently assigned random non-negative weights $\xi_v^{(i)} \sim \xi^{(i)}$. Given the weights, edges $vw$ appear in $H$ independently with probability $\min(\frac {\xi_v^{(1)} \xi_w^{(2)}} {\sqrt{n_1 n_2}}, 1)$.

We will also consider random intersection graphs $G^{conf}(d_1, d_2)$ based on the \emph{configuration model}, see \cite{guillaumelatapy2006, jansonluczak2009, wormald1999}.
    Let $d_1 = \{d_{1,u}, u \in V^1\}$ and $d_2 = \{d_{2,v}, v \in V^2\}$ be sequences
    of non-negative integers indexed by $V_1$ and $V_2$ respectively such that $\sum_u d_{1,u} = \sum_v d_{2,v}$.
    The 
    random bipartite multigraph
    $H^{conf}(d_1, d_2)$ with parts $(\Vn 1, \Vn 2)$ of sizes $n_1$ and $n_2$ is obtained as follows.
    Distribute the total number of $2 \sum d_{1,u}$ half-edges among the vertices of $\Vn 1 \cup \Vn 2$ so
    that the $j$-th vertex of part $i$, $v=v_j^{(i)}$, receives  $d_{i,v}$ half-edges.
    Pick a uniformly random perfect matching between the half-edges of parts $V^1$ and $V^2$.
    In the bipartite graph, add an edge between $u$ and $v$ whenever a half-edge from $u$
    is matched with a half-edge from $v$ (we allow multiedges).

    Usually, see, e.g., \cite{bloznelis2011}, the above models yield
    random graphs with a linear number of edges and a clustering coefficient bounded away from zero only if $\frac {n_2} {n_1} = \Theta(1)$. Therefore we will assume $\frac {n_2} {n_1}  = \Theta(1)$ in this paper.

    Let $\mu$ be the distribution of a random variable $Z$ on $[0, \infty)$ with $0 < \E Z < \infty$. We denote by $Z^*$ a random variable with the \emph{size-biased} distribution\footnote{We follow the star-notation of other authors, see, e.g., \cite{arratiagoldstein2009, lovasz-hom-book}. We also use symbols such as $G^*, d^*, v^*$ to denote objects unrelated to size-biased random variables, the actual meaning should be clear from the context.}
    \[
        \mu^*(A) = (\E Z)^{-1} \int_{A} t d \mu(t)
    \]
 for any Borel set $A$. If $Z$ is integer valued, then $\pr(Z^*=k) = (\E Z)^{-1} k \pr(Z=k)$.

Given two random variables $D_1, D_2$ on $\{0,1,2, \dots\}$ with $\E D_1, \E D_2 \in (0, \infty)$ 
define a multi-type Galton-Watson process $\ct = \ct(D_1, D_2)$ as follows.
$S(0)$ consists of a single root node $r=\rr(\ct)$.
The root $r$ has a set $S(1)$ of offspring, where $|S(1)| \sim D_1$.
For each $k \ge 1$, $S(k+1)$ consists of the offspring of the nodes in $S(k)$. 
Given $|S(k)|$, the number of offspring of each node in $S(k)$ is independent and distributed as $D_{i(k)}^*-1$. Here $i(k)=2$
if $k$ is odd and $i(k)=1$ otherwise. We call $S(k)$, the set of vertices at distance $k$ from the root, the \emph{generation} $k$ of $\ct$.
A corresponding random tree, also denoted by $\ct$, is a graph on the vertex set $\cup_k S(k)$ with edges $\{uv: v\mbox{ is an offspring of }u\}$ and root $r$. Consider $\ct$ as a bipartite graph with parts $(\Vn 1, \Vn 2)$, where $\Vn 1$ and $\Vn 2$ consists of all nodes in generations $0,2,\dots$ and $1,3,\dots$ respectively. We define the \emph{uncorrelated random clique tree}  $G_{\ct}$
to be the intersection graph of $\ct$ rooted at $r$.

For a finite (random) sequence $A$, we write $X \in_u A$ to denote the fact that $X$ is chosen uniformly at random from all the elements of $A$ (given $A$).
For random variables $Z, Z_1, Z_2, \dots$ we denote by $Z_n \xrightarrow{d} Z$ the fact that $Z_n$ converges in distribution to $Z$.

Let $H$ be a rooted connected graph. For a (multi-)graph\footnote{For multigraphs
    we define $G_1 \cong G_2$ if and only if there are bijections $\phi_1: V(G_1) \to V(G_2)$
 and $\phi_2: E(G_1) \to E(G_2)$ such that $\phi_1$ maps the endpoints of $e$ to the endpoints of $\phi_2(e)$
for each edge $e \in G_1$, and $\phi_1(\rr(G_1)) = \rr(G_2)$.}
$G$ of size $n_1 \ge 1$ write
$p_r(G, H) = n_1^{-1} |\{v \in V(G): B_r(G, v) \cong H\}|$. Let $\{G_n, n=1,2,\dots\}$ be a sequence of
finite random graphs with $|V(G_n)|\ge 1$, let $v^*_n \in_u V(G_n)$ and let $G^*$ be a random graph on $(\mathcal{G}_*, d_{loc})$.
We write $\cl((G_n,v^*_n)|G_n) \weaklypto \cl(G^*)$ as $n\to\infty$ if for each non-negative integer $r$ and each rooted connected graph $H$
    \begin{equation} \label{eq.weaklimitp}
        p_r(G_n, H) \xrightarrow{p} \pr(B_r(G^*) \cong H).
    \end{equation}
As observed in the recent literature (e.g. \cite{stufler2019}), this is equivalent to the convergence of the conditional random measures $\cl((G_n,v^*_n)|G_n)$ to the (fixed) measure $\cl(G^*)$ in probability, also known as \emph{quenched} convergence. (That is, consider the space of Borel measures on $(\mathcal{G}_*, d_{loc})$ with the L\'{e}vy-Prokhorov metric $\pi$. Then $\pi(\cl((G_n,v^*_n)|G_n), \cl(G^*)) \xrightarrow{p} 0$ if and only if (\ref{eq.weaklimitp}) holds for each $r$ and each $H$ as above. This can be seen using a similar argument as (iv) on p. 72 of Billingsley~\cite{billingsley_1999}.) While this equivalence reduces many questions related to local weak limits to the classical theory for separable metric spaces, in this paper it is only used to justify our notation.
\begin{theorem} \label{thm.riglocal}
    Let $\{G_n\}$ be a sequence of random intersection graphs where the underlying bipartite graphs are $H_n = (\Vn 1, \Vn 2, F)$ with
    $\Vn 1 = \Vn 1(n), \Vn 2 = \Vn 2(n)$ and $F = F(n)$. For $i=1, 2$ write $v_i^* = v_i^*(n)$ where $v_i^*(n) \in_u \Vn i$, $n_i = n_i(n) = |\Vn i|$ and $\Xn i = \Xn i(n)  = X_{v_i^*}$.

    Suppose $\{n_1\}, \{n_2\}$ are sequences of positive integers, such that $n_1,n_2 \to \infty$, $n_2/n_1 \to \beta \in (0, \infty)$ and
    \begin{enumerate}[(i)]
        \item either $G_n$, $n=1,2,\dots$ is an active random intersection graph and there is a random variable $D_1$ with $\E D_1 \in (0, \infty)$ such that $\E \Xn 1 \to \E D_1$ and
            \begin{equation}
                \Xn 1 \xrightarrow{d} D_1; \label{eq.cond1.1}                
            \end{equation}
        \item or $G_n$, $n=1,2,\dots$  is a passive random intersection graph and there is a random variable $D_2$ with $\E D_2 \in (0, \infty)$ such that $\E \Xn 2 \to \E D_2$ and
            \begin{equation}
                \Xn 2 \xrightarrow{d} D_2; \label{eq.cond1.2}                
            \end{equation}
        \item or $G_n = G^{inhomog}(n_1, n_2, \xi^{(1)}, \xi^{(2)})$, $n = 1,2,\dots$, 
              such that for $i=1,2$ $0 < \E \xi^{(i)} < \infty$ and $\xi^{(i)}$ does not depend on $n$;
          \item or $G_n = G^{conf}(d_1, d_2)$, $n = 1,2, \dots$, where $d_1 = d_1(n)$, $d_2 = d_2(n)$ 
              are non-random\footnote{The previous version of this manuscript \cite{kurauskas2015} stated an analogous result for random $d_1, d_2$. We simplified the condition to match, e.g., \cite{andersson1998,jansonluczak2009}, the previous result follows by a simple technical argument.}
              and for $i=1,2$ we have $\E d_{1,v_i^*} \to \E D_i$ and $d_{1,v_i^*} \xrightarrow{d} D_i$.
    \end{enumerate}
 
    Then both (\ref{eq.cond1.1}) and (\ref{eq.cond1.2}) hold and
    $\cl((G_n, v^*_1)|G_n) \weaklypto \cl(G_\ct)$ with $\ct = \ct(D_1, D_2)$ where any $D_i$ that is not defined here is defined in Remark~\ref{rmk.remark1}.

\end{theorem}
The proof is available in 
Appendix~\ref{sec.proofs}.
(In the case (iv) a previous version (v2) of \cite{kurauskas2015} stated an analogous result for random sequences $d_1, d_2$. We simplified the condition to match, e.g., \cite{andersson1998,jansonluczak2009}. The previous result follows by a simple technical argument.)

Recall that given a non-negative random variable $X$, a mixed Poisson random variable with parameter $X$ attains value $k$ with probability $\E e^{-X} X^k (k!)^{-1}$ for $k = 0, 1, \dots$. We denote this distribution by $Po(X)$. 

\begin{remark}\label{rmk.remark1}
    (See also \cite{bloznelis2010, bloznelis2011}.)
    In the case (i) we have $D_2 \sim Po(\beta^{-1} \E D_1)$, in the case (ii) we have $D_1 \sim Po(\beta \E D_2)$ and
    in the case (iii) we have $D_1 \sim Po(\beta^{1/2} \xi^{(1)} \E \xi^{(2)})$, $D_2 \sim Po( \beta^{-1/2} \xi^{(2)} \E \xi^{(1)})$.
    Thus in (i)-(iv) $\beta \E D_2 = \E D_1$.
\end{remark}

\begin{remark}\label{rmk.remark2}
For arbitrary random variables $D_1'$, $D_2'$ 
on $\{0, 1, 2, \dots \}$ with positive means,
there is a sequence of random configuration intersection graphs as in (iv)
for which $D_1 = D_1'$, $D_2 = D_2'$. 
\end{remark}
Thus in the active, the passive and the inhomogeneous model either $D_1$, or $D_2$, or both, has a (mixed) Poisson distribution.
The configuration model generalises these models in terms of local weak limits. For
example, both $D_1$ and $D_2$ can be power-law.

The fact that $\ct(D_1, D_2)$ is a limit for many sparse bipartite graph sequences is intuitive and in some physics literature has
been assumed implicitly \cite{karrernewman2010, newmanstrogatzwatts2001}. 
A result similar to Theorem~\ref{thm.riglocal}~(iv) can be found in
 \cite{bordenavelelarge2009, richardsonurbanke2008}. A nice proof for almost sure convergence in random configuration graphs (which could possibly be extended to bipartite graphs) can be found in \cite{dembomontanari2010}. For completeness, we provide our own formal proof in Appendix~\ref{sec.proofs}. 
We do not use the second moment condition and work under slightly weaker assumptions (convergence in probability). We are not aware of prior literature on weak limits in cases (i)-(iii). 

%

\section{Applications}
\label{sec.applications}

The proofs of the results in this section are given in Section~\ref{sec.applications.proofs}.
Let $\{G_n\}$ be a sequence of finite random graphs,
    and $G^*$ a random element on $(\mathcal{G}_*, d_{loc})$.
Assume $|V(G_n)| \ge 1$ for all $n$ and let $v^*_n$ be chosen uniformly at random from $V(G_n)$ (given $G_n$).

\subsection{Subgraph counts in random graphs}

To apply Theorem~\ref{thm.gencounts} in a random setting we need some easy technical facts.
Due to the equivalence mentioned after (\ref{eq.weaklimitp}), the L\'{e}vy-Prokhorov metric and Skorokhod's representation theorem could be used to show these or stronger properties
(see e.g., \cite{bordenavelelarge2009, stufler2019}), however we derive them from more basic arguments.

\begin{lemma}\label{lem.gencountsprob}
    Suppose $\cl((G_n,v^*_n)|G_n) \weaklypto \cl(G^*)$ and  $\{(G_n, v^*_n)\}$ are defined on the same probability space\footnote{Here this assumption is essential; in the rest of the paper it only affects the notation.}.
    Then there is a random set $A$ of positive integers such that 
    \begin{enumerate}[a)]
        {\setlength\itemindent{25pt}  
        \item $\pr(n \in A) \to 1$ as $n \to \infty$ and
        \item almost surely $|A|=\infty$ and $(G_n,v_n^*) \weaklyto G^*$  as  $n\to\infty$, $n \in A$.
        }
\end{enumerate}
\end{lemma}

\begin{lemma}\label{lem.measureconvp}
   $\cl((G_n,v^*_n)|G_n) \weaklypto \cl(G^*)$ if and only if for each bounded continuous function $f: (\mathcal{G}_*, d_{loc}) \to \mathbb{R}$ we have 
        $\E (f(G_n, v^*_n) | G_n) \xrightarrow{p} \E f(G^*)$.
\end{lemma}

We now restate Theorem~\ref{thm.gencounts} for sequences of random graphs.
\begin{lemma}\label{lem.gencountsp}
    Let $h\ge2$ be an integer,
    suppose $\cl((G_n,v^*_n)|G_n) \weaklypto \cl(G^*)$
    and assume $\inf |V(G_n)| \to \infty$. 
    As before, denote
    $d^* = d_{G^*}(r^*)$, $r^* = \rr(G^*)$, $n_1=n_1(n)=|V(G_n)|$ and assume $\E (d^*)^{h-1} < \infty$. 
    Denote by $\dn$ the degree of a uniformly random vertex in $G_n$.
    Then the following statements are equivalent:
    \begin{enumerate}[(i)]
        \item $\E \dn^{h-1} \to \E (d^*)^{h-1}$;  \label{eq.degmoment2}
        \item $\dn^{h-1}$ is uniformly integrable; \label{eq.uint2}
        \item for any connected graph $H$ on $h$ vertices and any $H' \in \cR(H)$  \label{eq.embconv2}
            \[
                n_1^{-1} \E \emb(H, G_n) \to \E \emb'(H', G^*, r^*).
            \]
    \end{enumerate}
     Each of the above statements implies that for any connected graph $H$ on $h$ vertices and any $H' \in \cR(H)$ 
    \begin{equation}  \label{eq.embconvp}
        n_1^{-1} \emb(H, G_n) \xrightarrow{p} \E \emb'(H', G^*, r^*).
    \end{equation}
\end{lemma}

\subsection{General weakly convergent sequences}
\label{subsec.app.gen}

In this section we assume that $\cl((G_n, v^*_n)|G_n) \weaklypto \cl(G^*)$, $n_1=n_1(n)= |V(G_n)|\ge3$ is non-random and $n_1 \to \infty$. 
As before, $d_n$ is the degree of $v^* = v^*_n$ in $G_n$ and $d^*$ is the degree of the root $r^*$ of $G^*$.
Lemma~\ref{lem.gencountsp} yields convergence of ${n_1}^{-1}\emb(H, G_n)$ provided that $(|V(H)|-1)$th degree moment of $G_n$ converges. This allows us to determine the limit behaviour of statistics based on subgraph counts. 


The \emph{clustering coefficient} of a graph $G$ is defined as
\[
    \alpha(G) := \frac {\emb(K_3, G)} {\emb(P_3, G)},
\]
where $K_3$ is the clique on 3 vertices and $P_t$ is the path on $t$ vertices.
(Set $\alpha(G) := 0$ when the denominator is zero.) 
For a rooted graph $H'$ let $\hom'(H', G, v)$
denote the number of homomorphisms from $H'$ to $G$ that map $\rr(H')$ to $v$.
For $t\ge 2$ let $K_t'$ be $K_t$ rooted at any vertex and let $K_{1,t}'$ be the bipartite graph $K_{1,t}$
rooted at the vertex of degree~$t$.
\begin{corollary}\label{cor.clust}
    Suppose $\cl( (G_n, v^*_n) | G_n) \weaklypto \cl(G^*)$ and 
     $\E \dn^2 \to  \E (d^*)^2 \in (0, \infty)$. Then
     \[
         \alpha(G_n) \xrightarrow{p}  \alpha^* := \frac{\E \emb'(K_3', G^*, r^*)} {\E \emb'(K_{1,2}', G^*, r^*)} = 
         \frac{\E \emb'(K_3', G^*, r^*)} {\E (d^*)_2}. 
     \]
\end{corollary}

The \emph{assortativity coefficient}, see e.g. \cite{bloznelis2013, litvakhofstad2013}, for a graph $G$ is defined as Pearson's correlation of the degrees 
 over the neighbouring vertices
 \[
     r(G) := \frac {g(G) - b(G)^2} {b'(G) - b(G)^2},
 \]
 where
 \begin{align*}
       g(G) &:= (2 e(G))^{-1}  \sum d_G(u) d_G(v); 
       & b(G) &:= (2 e(G))^{-1}    \sum d_G(u);
     \\ b'(G) &:= (2 e(G))^{-1} \sum d_G(u)^2; & e(G) &:= |E(G)|
 \end{align*}
 and the sums are over all $2 e(G)$ ordered pairs $(u,v)$ of adjacent vertices in $G$.
 We define $r(G) := 0$ when either $e(G)$ or $b'(G) - b(G)^2$ is zero (i.e., $G$ is regular).
 The above quantities can be easily expressed in terms of subgraph count statistics, see, e.g., \cite{bollobasjansonriordan2011} and Section~\ref{sec.applications.proofs}.
Denote by $P_4'$ the graph $P_4$ rooted at one of its internal vertices. 
\begin{corollary}\label{cor.assort}
    Suppose $\cl((G_n, v^*_n)|G_n) \weaklypto \cl(G^*)$,  $\E \dn^3 \to  \E (d^*)^3 < \infty$ and $Var(d^*) > 0$.
     Then
     \begin{equation}\label{eq.cor.assort}
         \E r(G_n) \xrightarrow{p} \rho^* :=  \frac{\E d^* \E \hom'(P_4', G^*, r^*) - (\E (d^*)^2)^2 }  {\E d^* \E (d^*)^3 - (\E (d^*)^2)^2}.
     \end{equation}
\end{corollary}
Corollaries~\ref{cor.clust} and \ref{cor.assort} easily follow from Lemma~\ref{lem.gencountsp}, see Section~\ref{sec.applications.proofs}.
Notice that since $\alpha(G), r(G) \in [0, 1]$, convergence in probability in these corollaries implies
convergence of means.

Statistics that can be expressed in terms of integrals of bounded functions,
such as the limit degree distribution, are obtained directly from the local weak limit. Hence
no degree moment conditions are necessary.
Let $\pi_k(G)$
be the fraction of vertices of degree $k$ in $G$. By Lemma~\ref{lem.measureconvp}
\begin{equation}\label{eq.degdistgen}
    \pi_k(G_n) = \E (\ii_{d_{G_n}(v^*_n)=k} | G_n) \xrightarrow{p}   \pr(d^* = k).
\end{equation}


Given a graph $G$ and an integer $k \ge 2$, let $(u_1^*, u_2^*, u_3^*)$ be 
a uniformly random triple of distinct vertices from 
$V(G)$.
The \emph{conditional clustering coefficient} is  
\[
    \alpha_k(G) := \pr(u_1^* u_3^* \in G | u_1^* u_2^*, u_2^*u_3^* \in G, d(u_2^*) = k),
\]
and set $\alpha_k(G) := 0$ if the event in the condition has probability zero.
Lemma~\ref{lem.measureconvp} implies that if $\pr(d^* = k) > 0$ then 
\begin{equation}\label{eq.cclust}
    \alpha_k(G_n) \xrightarrow{p} \alpha_k^* = \frac{\E \ii_{d^* = k} \emb'(K_3', G^*, r^*)} {k(k-1)\pr(d^* = k)}.
\end{equation}

The \emph{conditional assortativity}, see \cite{bloznelis2013} is defined as  
\[
    r_k(G) := \E(d_G (u_2^*) | u_1^* u_2^* \in G, d(u_1^*) = k),
\]
and set $r_k(G) := 0$ if the event in the condition has probability zero. Let $P_3'$ be $P_3$ rooted at one of the endpoints.
\begin{corollary}\label{cor.cassort}
    Suppose $\cl((G_n,v^*_n)|G_n) \weaklypto  \cl(G^*)$, $\E \dn^2 \to \E (d^*)^2 < \infty$ and $\pr(d^* = k) > 0$. Then
\begin{equation*}
    r_k(G_n) \xrightarrow{p} r_k^* = \frac{\E \ii_{d^* = k} \hom'(P_3', G^*, r^*)} {k \pr(d^* = k)} = 1 + \frac{\E \ii_{d^* = k} \emb'(P_3', G^*, r^*)} {k \pr(d^* = k)}.
\end{equation*}
\end{corollary}

In a similar way we can study the bivariate degree distribution \cite{bloznelis2014} and many other functionals.

\subsection{The case of random intersection graphs}

Here we apply the above general results in the case where the limit
    is the uncorrelated clique tree of Section~\ref{sec.intro}. We stress that Theorem~\ref{thm.gencounts} and its corollaries are applicable to a much broader class of sequences; including the inhomogeneous sparse random graph and the preferential attachment model \cite{bergerborgschayessaberi2014, bollobasjansonriordan2011}, general random configuration graphs with their many potential applications, see e.g. \cite{jansonluczak2009}, and random graphs from certain minor-closed classes, including random planar graphs \cite{georgakopauloswagner2015, psw2016, stufler2019}.

 Theorem~\ref{thm.riglocal} yields the first main condition (convergence to a local weak limit) for Lemma~\ref{lem.gencountsp}.
 For the other condition (convergence of a degree moment) we prove
\begin{lemma} \label{lem.rigdeg}
    Let $\{G_n\}$ be a sequence as in Theorem~\ref{thm.riglocal} and
    let $k$ be a positive integer.
    Suppose an additional condition for each of the cases (i)-(iv) of Theorem~\ref{thm.riglocal} holds:
       \begin{enumerate}[(i)]
        \item $\E (X^{(1)})^k \to \E D_1^k < \infty$;
        \item $\E (X^{(2)})^{k+1} \to \E D_2^{k+1} < \infty$;
        \item $\E (\xi^{(1)})^k < \infty$ and $\E (\xi^{(2)})^{k+1} < \infty$;
        \item 
            $\E d_{1,v_1^*}^k \to \E D_1^k < \infty$ and\footnote{We simplified (iv) of \cite{kurauskas2015} to fixed sequences and dropped a redundant assumption. To extend it to random $d_1, d_2$ use similar arguments as in the proof of Lemma~\ref{lem.gencountsp}.} $\E d_{2,v_2^*}^{k+1} \to \E D_2^{k+1} < \infty$. 
    \end{enumerate}
 Then $\E (d^*)^k < \infty$ and $\E \dn^k \to \E (d^*)^k$.
   \end{lemma}
\medskip
  The special case of (i) where $k \le 2$ was shown in \cite{blozneliskurauskas2013}. Here
  we use a different argument based on Theorem~\ref{thm.riglocal}, see Section~\ref{sec.applications.proofs}.

%


    Using the same notation as in Section~\ref{subsec.app.gen}, assume  that $\cl((G_n,v^*_n)|G_n) \weaklypto  \cl(G^*) = \cl(G_\ct)$ 
  where $\ct = \ct(D_1, D_2)$, $\E D_1 > 0$ and $\E D_2 > 0$. 
Let $Z_1, Z_2, \dots \sim D_2^*-1$ be independent and independent of $D_1$.
By (\ref{eq.degdistgen}) we have  
$\pi_k(G_n) \xrightarrow{p} \pr(d^* = k) = \pr(\sum_{i=1}^{D_1} Z_i = k)$.
For sequences of graphs as in Theorem~\ref{thm.riglocal}(i), (ii), (iii) the corresponding
convergence of means has been shown in \cite{bloznelis2008, bloznelisdamarackas2013}, see also
Remark~\ref{rmk.remark1}. 
We also notice that the second moment condition required in \cite{bloznelisdamarackas2013} for the inhomogeneous model
is not necessary.

By simple calculations we get 
\begin{align*}
    &\E Z_1^k = \frac {\E (D_2 -1)^k D_2 } {\E D_2}; \quad \E (Z_1)_k = \E (D_2)_{k+1} (\E D_2)^{-1}  \quad k = 1, 2, \dots;
 \\ &\E d^* = \E (Z_1 + \dots + Z_{D_1}) =  \E D_1 \E Z_1;  
 \\ &\E (d^*)^2 = \E D_1 \E Z_1^2 + \E (D_1)_2 (\E Z_1)^2;
 \\ &\E (d^*)^3 = \E D_1 \E Z_1^3  + 3 \E (D_1)_2 \E Z_1 \E Z_1^2 + \E (D_1)_3 (\E Z_1)^3;
 \\ &\E \emb'(K_3', G^*, r^*) = \E D_1 \E (Z_1)_2 = \E D_1 (\E Z_1^2 - \E Z_1)
\end{align*}
and
\begin{align} \label{eq.P4'}
    &\E \hom'(P_4', G^*, r^*) =  \E D_1 \E Z_1^3  + \E (D_1)_2 \E Z_1 \E Z_1^2 + \frac {\E (D_1)_2} {\E D_1} \E (d^*)^2 \E Z_1 .
\end{align}
(The above estimates hold also in the case when either side is infinite.)


 
When $\E D_2^3 < \infty$ and $\E D_1^2 < \infty$, we have
\[
   \alpha^* = \frac {\E D_1 \E D_2 \E (D_2)_3} {\E D_1 \E D_2 \E(D_2)_3 + \E (D_1)_2 (\E (D_2)_2)^2}.
\]
in Corollary~\ref{cor.clust}. Using Remark~\ref{rmk.remark1},
this simplifies to $\alpha^* = \E D_1 / \E D_1^2$ for active 
and to $\alpha^* = \E (D_2)_3 (\E (D_2)_3 + \beta (\E (D_2)_2)^{-2}$ for passive random intersection graphs.
This is equal to a related estimate $\hat{\alpha} = \lim \E \emb(K_3, G_n) (\E \emb(P_3, G_n))^{-1}$
obtained by Bloznelis \cite{bloznelis2011} and the estimates of Godehardt, Jaworski and Rybarczyk \cite{gjr2012}
for these particular models. 

Similarly, if $\E D_1^2 < \infty$ and $\E D_2^4 < \infty$ then $\rho^*$ in Corollary~\ref{cor.assort} is a rational function of $\E D_1$, $\E D_1^2$ and $\E D_2^k$, $k=1,2,3,4$ obtained by using (\ref{eq.P4'}) and the above expressions for $\E (d^*)^j$ in (\ref{eq.cor.assort}). 
One can check by simple algebra that $\rho^*$ is equal
to $\hat{\rho} = \lim (\E g(G_n) - \E b(G_n)^2) (\E b'(G_n) - \E b(G_n)^2)^{-1}$ computed
in \cite{bloznelis2013} for sparse passive and active random intersection graphs.

Assuming only that 
$\pr(d^* = k) > 0$ we get in (\ref{eq.cclust})
\begin{equation} \label{eq.condclust.1}
    \alpha_k^* = \frac { \E (\sum_{j=1}^{D_1} Z_i (Z_i - 1) | d^* = k)  } { k (k-1)}.
\end{equation}
If $D_2 \sim \Po(\lambda)$ as is the case, e.g.,  for the active random intersection graph of Theorem~\ref{thm.riglocal}(i), then
as in \cite{bloznelis2013} (but without a second moment assumption)
\begin{equation}\label{eq.condclust.2}
   \alpha_k^* = \frac {\lambda \pr(d^* = k-1)} {k \pr(d^* = k)}.
\end{equation}

Finally, if $\E D_1^2 < \infty$, $\E D_2^2 < \infty$ in Corollary~\ref{cor.cassort} we have
\[
    r_k^* = k^{-1} \E \left( \sum_{i=1}^{D_1} Z_i^2 \bigg| d^* = k \right) + \frac {\E (D_1)_2 \E (D_2)_2} {\E D_1 \E D_2}.
\]
This agrees with a related estimate obtained for active and passive random intersection graphs in \cite{bloznelis2013}.

Thus Corollaries~\ref{cor.clust}--\ref{cor.cassort}
generalise several previous results for particular random intersection graph models
to arbitrary sequences of graphs with the uncorrelated clique tree as a limit.
Applying them together with
Lemma~\ref{lem.rigdeg} with an appropriate $k$ yields slightly stronger versions (i.e., convergence in probability and optimal moment conditions) 
of these results
for the active and passive random intersection graphs with bounded expected degree. 
We are not aware of similar prior results for the inhomogeneous and configuration models.



\subsection{Proofs}
\label{sec.applications.proofs}

The \emph{radius} of a connected rooted graph is the maximum distance
from any vertex of the graph to the root.

\begin{proofof}{Lemma~\ref{lem.gencountsprob}}
    Let $H_1, H_2 \dots$ be an enumeration of finite graphs in $\mathcal{G}_*$. 
    For positive integers $i, n$ define the event 
    \[
        B(i,n) = \{\exists j \le i: \left | p_{r_j}(G_n, H_j) - \pr( B_{r_j}(G^*) \cong H_j) \right | > i^{-1} \},
    \]
    where $r_j$ is the radius of $H_j$. By the assumption of the lemma $\pr(B(i,n)) \to 0$ for each $i=1,2,\dots$.
    Define $N_1 = 1$, $N_i, i = 2,3,\dots$ by taking 
    $N_i = 1 + \sup \{n > N_{i-1} : \pr(B(i,n)) > i^{-1}\}$. Let $i(n) = \max\{i : N_i \le n\}$.
    Now let $A = \{n: \overline{B(i(n),n)} \}$.
    We have $\pr(n \not \in A) = \pr(B(i(n), n)) \le i(n)^{-1} \to 0$.
    For any sequence of events $\{A_n, n\ge1\}$ on the same probability
    space $(\Omega, \cf, \pr)$ we have
    \begin{equation}\label{eq.limsupA}
        \pr(\cap_{m \ge 1} \cup_{n\ge m} A_n) \ge \lim \sup \pr(A_n).
    \end{equation}
    Thus $\pr(|A| = \infty) \ge \lim \sup \pr(\overline{B(i(n),n)}) = 1$. 
    Now by the definition of $B(i,n)$ on the event $|A| = \infty$ we have $(G_n, v_n^*) \weaklyto G^*$ as $n\to\infty$, $n \in A$.
\end{proofof}

\bigskip

\begin{proofof}{Lemma~\ref{lem.measureconvp}}
    ($\Leftarrow$) The function $(\mathcal{G}_*, d_{loc}) \to \mathbb{R}$ that maps $(G,v)$ to $\ii_{B_r(G,v) \cong H}$ is bounded and continuous for each $r\ge 0$ and connected rooted graph~$H$.

    ($\Rightarrow$) Without loss of generality we may assume $\{(G_n,v_n^*)\}$ are defined on a single probability space.
   Let $A$ be a random set guaranteed by Lemma~\ref{lem.gencountsprob}. Suppose there is some $\eps > 0$, 
   a bounded continuous function $f$
   and an infinite subset of positive integers $B$, such that 
    $\pr( |\E(f(G_n,v_n^*)|G_n) - \E f(G^*)| > \eps) > \eps$ for all $n \in B$.
    Define a random set $C = \{n \in B: |\E(f(G_n,v_n^*)|G_n) - \E f(G^*)| > \eps\}$. Since for $n \in A \cap B$ we have $\pr(n \in A \cap C) \ge \eps - o(1)$, by
    (\ref{eq.limsupA}) $\pr(|A \cap C| = \infty) \ge \eps$. However, $(G_n,v_n^*) \weaklyto G^*$ when $n \to \infty, n \in A$,
    by (\ref{eq.measureconv}) this is a contradiction to our assumption.
\end{proofof}

\medskip

\begin{proofof}{Lemma~\ref{lem.gencountsp}}
    We can assume $n_1 \ge 1$.

    \textit{ (\ref{eq.degmoment2}) $\Leftrightarrow$  (\ref{eq.uint2})}.
    This follows by Lemma~\ref{lem.uint} 
    since Lemma~\ref{lem.measureconvp} implies convergence in distribution
    of $d_n^{h-1}$ to $(d^*)^{h-1}$.

    \textit{ (\ref{eq.degmoment2}) $\Rightarrow$  (\ref{eq.embconvp}), (\ref{eq.embconv2})}.
    Assume (\ref{eq.degmoment2}).  
    Then there is a positive sequence $a_n \to 0$,
    such that $|\E d_n^{h-1} - \E (d^*)^{h-1}| \le a_n$.
    The empirical $(h-1)$-st moment of the degree of $G_n$ is
    $D_n = n_1^{-1} \hom(K_{1,h-1}, G_n)$.  
    Also $D_n = \sum_{H \in \mathcal{S}} d_{H}(\rr(H))^{h-1} p_r(G_n, H)$,
    where $\mathcal{S}$ consists of graphs in $\mathcal{G}_*$ of radius $1$.
Since $\E D_n = \E d_n^{h-1}$ it follows
    by (\ref{eq.degmoment2}) and  (\ref{eq.weaklimitp})  that $D_n \xrightarrow{p} \E(d^*)^{h-1}$. 
    So there is a positive sequence $\eps_n \to 0$,  such that for all $n$
    \begin{equation*} 
        \pr(|D_n - \E (d^*)^{h-1}| > \eps_n) \le \eps_n.
    \end{equation*}
    We may assume that $\eps_n \ge a_n$.
    Let the random set $C$ consist of those $n$ for which $|D_n - \E (d^*)^{h-1}| \le \eps_n$. It follows by (\ref{eq.limsupA}) 
    that $\pr(|C| = \infty) = 1$, $\pr(n \in C) \to 1$ and
    on the event $|C| = \infty$, 
    $D_n \to \E (d^*)^{h-1}$
    as $n \to \infty$, $n \in C$.

    Let $A$ be the random set guaranteed by Lemma~\ref{lem.gencountsprob}. On the event $|A \cap C| = \infty$,
    the subsequence of graphs $\{G_n, n \in A \cap C\}$ satisfies the conditions of 
    Theorem~\ref{thm.gencounts}.
    
    Assume (\ref{eq.embconvp}) does not hold.
    Then there is $H' \in \mathcal{R}(H)$, $\eps > 0$
    and a deterministic infinite set $D$ of positive integers such that for all $n \in D$
    \[
        \pr( |X_n  - \E X^*| > \eps) > \eps.
    \]
    Here $X_n = n_1^{-1} \emb(H, G_n)$ and $X^* = \emb'(H', G^*, r^*)$.
    Let $D_1 \subseteq D$ consist of those $n$ in $D$ for which $|X_n  - \E X^*| > \eps$.
    Again, by (\ref{eq.limsupA}),  we get that $\pr(|A\cap C \cap D_1| = \infty) \ge \epsilon$.
    On this event $X_n \not \to \E X^*$ as $n \to \infty, n \in A \cap C$.
    This is a contradiction to Theorem~\ref{thm.gencounts}~(\ref{eq.embconv}).

    It remains to show (\ref{eq.embconv2}). 
    Write  $Y_n = n_1^{-1} \hom(H, G_n)$.
    Trivially, $X_n \le Y_n$, and by Lemma~\ref{lem.sidorenko}
    $Y_n \le D_n$. So for any $t > 0$
    \[
        \E X_n \ii_{X_n > t} \le \E Y_n \ii_{Y_n > t} \le \E D_n \ii_{D_n > t}.
    \]
    Since $\E D_n \to \E (d^*)^{h-1}$ and $D_n \xrightarrow{p} \E (d^*)^{h-1}$, $D_n$ is uniformly integrable, 
    and so is $X_n$.
    Since $X_n$ also converges in probability by (\ref{eq.embconvp}), (\ref{eq.embconv2}) follows by Lemma~\ref{lem.uint}.

    \emph{(\ref{eq.embconv2})$\Rightarrow$(\ref{eq.degmoment2}).} The proof is identical to that
    of the corresponding implication of Theorem~\ref{thm.gencounts}.
\end{proofof}

\medskip

\begin{proofof}{Corollary~\ref{cor.clust}}
    Apply Lemma~\ref{lem.gencountsp}.
\end{proofof}

\medskip




\begin{proofof}{Corollary~\ref{cor.assort}}
    For non-empty $G$ we have
    \begin{align*}
        &g(G) = \frac {\hom(P_4, G)} {\emb(K_2, G)};
        &
        &b(G) = \frac {\hom(K_{1,2}, G)} {\emb(K_2, G)};
        &
        &b'(G) = \frac {\hom(K_{1,3}, G)} {\emb(K_2, G)}.
    \end{align*}
    Let $S(t,j)$ denote Stirling numbers of the second kind. 
    Using Lemma~\ref{lem.gencountsp},
    \begin{align*}
        &n^{-1} \emb(K_2, G_n) \xrightarrow{p} \E d^*;
        \\    &n^{-1} \hom(P_4, G_n) = n^{-1} \left(\emb(P_4, G_n) + \emb(K_3, G_n) + 2 \emb(P_3, G_n) +  \emb(K_2, G_n) \right)
        \\ &
        \xrightarrow{p} \E \left( \emb'(P_4', G^*) + \emb'(K_3', G^*) + \emb(K_{1,2}', G^*) + \emb(P_3', G^*) + \emb'(K_2', G^*)\right) 
        \\& =  \E \hom'(P_4', G^*,r^*);
        \\ &n^{-1} \hom(K_{1,t}, G_n) = n^{-1} \sum_{j=1}^t S(t, j) \emb(K_{1,j}, G_n) \xrightarrow{p} \E \hom'(K_{1,t}', G^*,r^*) = \E (d^*)^t
    \end{align*}
    for $t = 2, 3$. The claim follows by the definition of $r(G)$. 
\end{proofof}

\medskip





   \medskip 

\begin{proofof}{Corollary~\ref{cor.cassort}}
   Note that $\pi_k(G_n) \xrightarrow{p} \pr(d^*=k) > 0$ and when $\pi_k(G) > 0$ we have
   \begin{align}
       r_k(G) &= \frac{\E \left( d_{G}(u_2^*) \ii_{d_{G_n}(u_1^*)=k} \ii_{u_1^*u_2^* \in G} | G_n=G \right)} {\pr(d_{G_n}(u_1^*) = k, u_1^*u_2^* \in G | G_n = G)} \nonumber
       \\   &=\frac{(n_1)_2^{-1} H(G)} {k (n_1-1)^{-1} \pi_k(G)} = n_1^{-1} H(G) (k \pi_k(G))^{-1}, \label{eq.cassort.sort}
   \end{align}
   where $H(G)$ is the number of homomorphisms from $P_3 = x y z$ to $G$ so that $x$ is mapped to a vertex of degree $k$.    Denote by $H_t(G)$ the number of such homomorphisms where additionally $y$ is mapped to a vertex of degree at most $t$,
   and let ${\bar H}_t(G) = H(G) - H_t(G)$.

   Fix  $\delta > 0$. We will show that for any $\eps > 0$ and all $n$ large enough
   \begin{equation}\label{eq.cassort.toshow}
   \pr(|n_1^{-1}H(G_n) - \E \ii_{d^*=k}\hom'(P_3', G^*, r^*) | > \delta) \le \eps,
   \end{equation}
   i.e. $n_1^{-1}H(G_n) \xrightarrow{p} \E \ii_{d^*=k}\hom'(P_3', G^*, r^*)$.

   By Lemma~\ref{lem.measureconvp}
   \begin{align*}
       &n_1^{-1} H_t(G_n) = \E\left(\ii_{d_{G_n}(u_1^*)=k} \sum_{u: u u_1^* \in G_n} d_{G_n}(u) \ii_{d_{G_n}(u) \le t} \Big | G_n\right)
   \\  &\xrightarrow{p}  h_t^* = \E \left(\ii_{d^*=k} \sum_{u: u r^* \in G^*} d_{G^*}(u) \ii_{d_{G^*}(u) \le t}\right).
   \end{align*}
    Also $h_t^* \to h^* = \E \ii_{d^* = k} \hom'(P_3', G^*, r^*)$ as $t \to \infty$ since $\cl((G_n, v^*_n)|G_n) \weaklypto \cl(G^*)$ and $h^* \le \E \hom'(P_3', G^*, r^*) < \infty$ by
   Lemma~\ref{lem.gencountsp}.
   Therefore we can pick $t_1$ such that for $t \ge t_1$ and all $n$ large enough
   \[
       \pr(|n_1^{-1} H_{t}(G_n) - h_{t}^*| > \frac \delta 4) \le \frac \eps 4; \quad |h_{t}^* - h^*| \le \frac \delta 4
   \]
   and so
   \begin{equation}\label{eq.split1}
       \pr(|n_1^{-1} H_{t}(G_n) - h^*|> \frac \delta 2) \le \frac \eps 2.
   \end{equation}
   Next, note that $\bar{H}_t(G_n) \le \sum_{v \in V(G_n)} d_{G_n}(v)^2 \ii_{d_{G_n}(v) > t}$. So by Markov's inequality
   \[
       \pr(n_1^{-1} {\bar H}_t(G_n) > \delta/2) \le 2 \delta^{-1} n_1^{-1} \E {\bar H}_t(G_n) \le 2 \delta^{-1} \E \dn^2 \ii_{\dn > t}.
   \]
   $\dn^2$ is uniformly integrable by Lemma~\ref{lem.gencountsp}, so there is $t_2$ such that for all $t \ge t_2$
   and all large enough $n$
   \begin{equation}\label{eq.split2}
       \pr(n_1^{-1} {\bar H}_{t}(G_n) > \delta/2) \le \frac \eps 2.
   \end{equation}
   Now (\ref{eq.cassort.toshow}) follows by setting $t = \max(t_1, t_2)$ and combining (\ref{eq.cassort.sort}), (\ref{eq.split1}) and (\ref{eq.split2}).
\end{proofof}

\medskip

\begin{proofof}{Lemma~\ref{lem.rigdeg}}
     Let $Z_1, Z_2, \dots \sim D_2^* - 1$ and $D_1$ be independent.
     For a random variable $X$ with $\E X \in (0, \infty)$ and its size-biased version $X^*$ we have 
     \[
         \E (X^*-1)_j = \sum_{m \ge 1} (m-1)_j \frac{m \pr(X=m)} {\E X}  = \frac{ \E (X)_{j+1}} {\E X}, \quad j = 1, 2, \dots 
     \]
     Thus using the assumptions and Remark~\ref{rmk.remark1} 
     \begin{equation} \label{eq.factD2}
         \E (Z_1)_j = \E (D_2)_{j+1} (\E D_2)^{-1} < \infty \mbox{ for } j=1, \dots, k.
     \end{equation}
     Similarly $\E Z_1^k = (\E D_2)^{-1} \E D_2^{k+1} < \infty$. 
     Write $\binom k {k_1, \dots, k_j} = \frac {k!} {k_1! \dots k_j!}$.
     Conditioning on $D_1$, using linearity of expectation and symmetry,
     we get
     \begin{align}
         \E (d^*)^k &= \E (\sum_{i=1}^{D_1} Z_i)^k 
         \nonumber
         \\ &
         = \sum \binom {k} {k_1, \dots, k_j} \E \binom {D_1} {j}  \E Z_1^{k_1} \times \dots \times \E Z_j^{k_j} \in (0,\infty).
         \label{eq.d*k}
     \end{align}
     Here the sum is over all $j$ and all tuples of positive integers $(k_1, \dots, k_j)$ such that $k_1 + \dots + k_j = k$.
     Write $\dn = d_{G_n}(v^*_n)$ and recall that $v^*_n$ is a uniformly random vertex from $V(G_n)$.
     By Theorem~\ref{thm.riglocal} $\cl((G_n,v^*_n)|G_n) \weaklypto  \cl(G^*)$, so $\dn^k \xrightarrow{d} (d^*)^k$.
     By Fatou's lemma
     \[
         \E (d^*)^k \le \lim \inf \E \dn^k.
     \]
     We assume without loss of generality that in the case (iv)
     the sequences $d_1(n)$ and $d_2(n)$ (not to be confused with the random variable $d_n$) are symmetric random permutations of two fixed degree sequences (each permutation of
     a particular sequence is equally likely).
     So in all cases (i)-(iv) by symmetry $\E d_{G_n}(v_1)^k = \E \dn^k$, where $v_1$
     is a fixed vertex in $V(G_n)$.
     For each of the random intersection
     graph models we will show
     \begin{equation} \label{eq.rigdeg.upper}
         \E d_{G_n}(v_1)^k \le \E (d^*)^k + o(1).
     \end{equation}
     Let $\ct \sim \ct(D_1,D_2)$. Assume that $D_1 = d_\ct(\rr(\ct))$,  $x_1, \dots, x_{D_1}$
     are the children of $\rr(\ct)$
     and $Z_i$ is the number of children of $x_i$.
     For each $n$ define a bipartite graph (a tree) $\tilde{H}_n$ as follows. 
     On the event $D_1 > n_2$, let $\tilde{H}_n$ be a tree consisting of just the root $\tilde{v}_1$.
     On the event $D_1 \le n_2$, let $\tilde{H}_n$ be the subtree induced
     by generations $0, 1$ and $2$ of $\ct$, but take only the first
     $Z_i' = Z_i \ii_{Z_i \le n_1 -1}$ children for the node $x_i$, $i = 1, \dots, D_1$.
     Label the root $v_1$. Given $D_1, Z_1, \dots, Z_{D_1}$
     draw labels for $x_1, \dots, x_{D_1}$ from $\Vn 2$ uniformly at random without replacement
     and draw $Z_i'$ distinct labels from $\Vn 1 \setminus \{v_1\}$ for the children of $x_i$ $i=1,\dots,D_1$, for each $i$ (conditionally) independently.
     Here $\Vn i = V_i(H_n)$ is the set of first $n_i$ vertices of the fixed ground set $\cv^i$ as in Theorem~\ref{thm.riglocal}.

     Write $\tilde{d}_n = \ii_{D_1 \le n_2} \sum_{i=1}^{D_1} Z_i'$ and notice that ${\tilde d}_n$ is an upper bound on the degree of $v_1$ in the resulting
     intersection graph. We have as in (\ref{eq.d*k})
     \begin{align}
         \E {\tilde d}_n^k 
         \nonumber
         &= \sum \binom k {k_1, \dots, k_j} \E \binom {D_1} j  \ii_{D_1 \le n_2} \E (Z_1')^{k_1}\times \dots \times \E (Z_t')^{k_t}
         \\ &= \E (d^*)^k - o(1). \label{eq.tdk}
     \end{align}
     Here we used (\ref{eq.factD2}),  (\ref{eq.d*k}) and bounds
     \[
         \E (D_1)_j \ii_{D_1 \le n_2} = \E (D_1)_j - o(1); \quad \E (Z_1')^j = \E Z_1^j - \E Z_1^j \ii_{Z_1 > n_1 -1} = \E Z_1^j - o(1),
     \]
     valid for any $j \le k$ by Lemma~\ref{lem.uint}. 
     Thus it suffices to prove that $\E d_{G_n}(v_1)^k \le \E {\tilde d}_n^k + o(1)$. Recall that $H_n$ is the bipartite
     graph underlying the intersection graph $G_n$. 
     Call a path  $xyz$ \emph{good} if $x = v_1$, $y \in \cv_2$ and $z \in \cv_1 \setminus \{v_1\}$. We have
    \[
        d_{G_n}(v_1) \le \sum \ii(v_1 w v, H_n) \quad\mbox{and}\quad \quad {\tilde d}_n = \sum \ii(v_1 w v, {\tilde H}_n)
    \]
    where the sum is over all 
    good paths $v_1 w v$ and $\ii(F, H)$ is the indicator
    of the event that $F \subseteq E(H)$.

    For any graph $H$ we denote $v(H) = |V(H)|$ and $e(H) = |E(H)|$. If $H$ is bipartite (more precisely, 2-coloured) 
    $v_j(H)$, $j=1, 2$ denotes the size of $j$-th part $V_j$ of $H$. 
    Define an equivalence relation between bipartite graphs $H' = (V_1', V_2', E')$, 
    $H'' = (V_1'', V_2'', E'')$: $H' \sim H''$ if and only if there is an isomorphism from $H'$
    to $H''$ that maps $V_j'$ to $V_j''$, $j=1,2$.
    Let $\cf_k$ consist of one member for each equivalence class of
    all graphs formed from a union of $k$ good paths
    (a not necessarily disjoint \emph{union} of graphs $G_1 = (V_1, E_1)$, $\dots$, $G_k = (V_k, E_k)$
    is a graph $(V_1 \cup \dots \cup V_k, E_1 \cup \dots \cup E_k)$). For $H' \in \cf_k$, let $N(H')$
    be the number of distinct tuples of $k$ good paths whose union is a bipartite graph $H''$ with parts $V_1'' \subseteq \Vn 1$ and $V_2'' \subseteq \Vn 2$ such that
    $H'' \sim H'$. It is 
    easy to see that there are positive constants $c(H'), C(H')$, such that for all $n$ large enough
    \begin{equation}\label{eq.cH}
        N(H') = c(H') (n_1)_{v_1(H')-1} (n_2)_{v_2(H')} = C(H') n_1^{v(H') - 1} (1 + o(1)).
    \end{equation}
    By linearity of expectation
    \begin{align*}
        &\E d_{G_n}(v_1)^k \le \E (\sum \ii(v_1 w v, H_n))^k =\sum_{H' \in \cf_k} N(H') \E \ii(H', H_n).
    \end{align*}
    and similarly
    \[
        \E {\tilde d}_n^k = \sum_{H' \in \cf_k} N(H') \E \ii(H', {\tilde H}_n).
    \]
    Using (\ref{eq.tdk}),  (\ref{eq.cH}) and the fact that $\cf_k$ is finite, in order
    to prove (\ref{eq.rigdeg.upper}) it suffices to check that
    \begin{equation}
    \label{eq.iiH}
        \E \ii(H', H_n) \le \E \ii(H', \tilde{H}_n) + o (n_1^{-v(H')+1}) \quad \mbox{ for each } H' \in \cf_k.
    \end{equation}
    So fix any $H' \in \cf_k$. 
    Suppose $v_2(H') = t$ and the degrees 
    of vertices in $V_2(H')$ are $b_1, \dots, b_t$. Note that $b_j \le k+1$ for $j = 1, \dots, t$.
    Conditioning on $D_1$ and the positions of generation 1
    nodes labelled $V_2(H')$ and using (\ref{eq.factD2}) 
    \begin{align}
        &\E \ii(H', {\tilde H}_n) = \E \frac {(D_1)_t \ii_{D_1 \le n_1}} {(n_2)_t} \frac{(Z_1')_{b_1-1}} {(n_1 - 1)_{b_1-1}}  \times \dots \times
        \frac{(Z_t')_{b_t-1}} {(n_1 - 1)_{b_t-1}} \nonumber
        \\ &= n_2^{-t} n_1^{-e(H') + t} \E (D_1)_t \prod_{i=1}^t \E (Z_i')_{b_i-1} (1+o(1)) \nonumber
        \\ &= \beta^{-t} n_1^{-e(H')} \E (D_1)_t (\E D_2)^{-t} \prod_{i=1}^t \E (D_2)_{b_i} (1+o(1)). \label{eq.tH}
    \end{align}
    Now if $H'$ is a tree then $e(H') = v(H')-1$ and (\ref{eq.iiH}) follows if
    \begin{equation} \label{eq.iiH1}
         \E \ii(H', H_n) \le \E \ii(H', \tilde{H}_n) (1 + o(1)).
    \end{equation}
    Meanwhile, if $H'$ has a cycle then $e(H') \ge v(H')$ and
    $\E \ii(H', {\tilde H}_n) = O(n_1^{-v(H')})$, so (\ref{eq.iiH})
    follows whenever
    \begin{equation} \label{eq.iiH2}
        \E \ii(H', H_n) = o(n^{-v(H')+1}).
    \end{equation}
    We now consider (\ref{eq.iiH}) for each model separately.

    \emph{(i) (active intersection graph)}
    Let $a_1, \dots, a_s$ be the degrees of vertices in $V_1(H')$.
    We can assume $a_1 = d_{H'}(v_1) = t$.
    Of course, $a_j \le k$, $j = 1, \dots, s$. Since the vertices in $V_1(H')$ choose
    their neighbours independently, using Lemma~\ref{lem.uint}
    \[
        \E \ii(H', H_n) = \E \prod_{i=1}^s \frac { (X_v)_{a_i}} { (n_2)_{a_i}} = n_2^{-e(H')} \prod_{i=1}^s  \E (D_1)_{a_i} (1 + o(1)).
    \]
    If $H'$ has a cycle then $e(H') \ge v(H')$ and $\E \ii(H', H_n) = O \left(n_1^{-e(H')}\right)$, so (\ref{eq.iiH2}) holds.

    By Remark~\ref{rmk.remark1}, $D_2 \sim Po(\beta^{-1} \E D_1)$. So $\E (D_2)_{b_i} = (\beta^{-1} \E D_1)^{b_i}$.
    Thus (\ref{eq.tH}) reduces to
    \[
        \E \ii(H', {\tilde H}_n) = (\beta n_1)^{-e(H')} \E (D_1)_t  (\E D_1)^{e(H') - t} (1 + o(1))
    \]
    If $H'$ has no cycle, then $a_j = 1$ for all $j \ge 2$. Thus
    \begin{align*}
        \E \ii(H', H_n) &\le 
        n_2^{-e(H')} \E (D_1)_t (\E D_1)^{e(H')-t} (1 + o(1))
    \end{align*}
    and (\ref{eq.iiH2}) follows.

    \emph{(ii) (passive intersection graph)} Since $b_i \le k+1$ for $i = 1, \dots, t$ by the assumption (ii) of the lemma
    \[
        \E \ii(H', H_n) = \E \prod_{i=1}^s \frac { (X_v)_{b_i}} {(n_1)_{b_i}} = n_1^{-e(H')} \prod_{i=1}^s \E (D_2)_{b_i} (1 + o(1)). 
    \]
    Using Remark~\ref{rmk.remark1}, $D_1 \sim Po(\beta \E D_2)$, so $\E (D_1)_t = \beta^t (\E D_2)^t$. Therefore (\ref{eq.tH})
    reduces to
    \[
        \E \ii(H', {\tilde H}_n) =  n_1^{-e(H')} \prod_{i=1}^s \E (D_2)_{b_i} (1 + o(1))
    \]
    and (\ref{eq.iiH}) follows.

    \emph{(iii) (inhomogeneous random intersection graph) } Let $\{\xi_u: u \in V(H')\}$ be independent random variables
    such that $\xi_u \sim \xi^{(i)}$ for $u \in V_i(H')$, $i = 1, 2$. Write $a \wedge b = \min(a,b)$. Then
    \begin{align*}
        \E \ii(H', H_n) &= \E \prod_{uv \in E(H')}\left( \frac {\xi_u \xi_v} {\sqrt{n_1 n_2}} \wedge 1\right) 
       \\               &\le \beta^{-e(H')/2}n_1^{-e(H')} \prod_{u \in V(H')} \E \xi_u^{d_{H'}(u)} (1 + o(1)).
    \end{align*}
    If $H'$ contains a cycle, then by the assumption that $\E (\xi^{(1)})^k$ and $\E (\xi^{(2)})^{k+1}$ are finite, we get that $\E \ii(H', H_n) = O\left(n_1^{-v(H')}\right)$, so (\ref{eq.iiH2}) holds. If $H'$ is a tree then
    \begin{equation}\label{eq.silhuette}
        \E \ii(H', H_n) \le \beta^{-e(H')/2} n_1^{-e(H')} \E (\xi^{(1)})^t (\E \xi^{(1)})^{s-1} \prod_{j=1}^t \E (\xi^{(2)})^{b_j} (1 + o(1)).
    \end{equation}
    Using Remark~\ref{rmk.remark1}, we have $D_1 \sim Po(\beta^{1/2} \xi^{(1)} \E \xi^{(2)})$ and  $D_2 \sim Po(\beta^{-1/2} \xi^{(2)} \E \xi^{(1)})$, so $\E (D_1)_t = \beta^{t/2} \E (\xi^{(1)}) ^ t (\E \xi^{(2)})^t$ and $\E (D_2)_j = \beta^{-j/2} \E (\xi^{(2)}) ^ j (\E \xi^{(1)})^j$ for $j \le k+1$. Putting these estimates into (\ref{eq.tH}) and simplifying we get the expression on the right of (\ref{eq.silhuette}).
    (\ref{eq.iiH1}) follows.

   \emph{(iv) (random configuration graph) }  For $i = 1, 2$, let
   \[
       \tilde{d}_{i, m} = n_i^{-1} \sum_{v \in \Vn i} (d_{i,v})_m.
   \]

   Recall that $N = \sum_{u \in \Vn 1}^{n_1} d_{1,u}$ is the total number of half-edges in each of the parts. Since $n_1, n_2 \to \infty$ and  by the assumption of the lemma $N n_1^{-1} \to \E D_1$, 
   there exists $\omega_n \to \infty$, such that for all $n$
   \begin{equation}\label{eq.omegan}
       n_1,n_2, N \ge \omega_n; \quad \omega_n \ge k+2.
   \end{equation}
   The probability that $H_n$ contains $H'$ as a subgraph is at most
   \[
       a(H') =  \frac {1} {(N)_{e(H')}} \E \prod_{u \in V(H')} (d_{H_n}(u))_{d_{H'}(u)}.
   \]
   Here the product counts the number of ways to choose particular half-edges forming $H'$.
   Let $(u_1^*, \dots, u_s^*)$ and $(w_1^*, \dots, w_t^*)$ be independent uniformly random tuples of distinct vertices
   from $V_1(H_n)$ and $V_2(H_n)$ respectively.  Using symmetry, 
   (\ref{eq.omegan}) and the assumption of the lemma
   \begin{align*}
       &a(H') =  \frac {1} {(N)_{e(H')}} \E \prod_{i=1}^s (d_{1,u_i^*})_{a_i} \prod_{j=1}^t (d_{2,w_j^*})_{b_j}
       \\ & \le \frac {1} {(N)_{e(H')}} \prod_{i=1}^s \tilde{d}_{1, a_i} \prod_{j=1}^t \tilde{d}_{2, b_j}(1 + o(1))
       \\& = (\E D_1 n_1)^{-e(H')} \prod_{i=1}^s \E (D_1)_{a_i} \prod_{j=1}^t \E (D_2)_{b_j} (1+o(1)).
   \end{align*}
   Again, if $H'$ has a cycle then (\ref{eq.iiH2}) follows. 
   Otherwise,
   if $H'$ is a tree, then since $\E D_1 n_1 = \E D_2 n_2 (1+o(1))$ we have $\E D_1 =  \beta \E D_2$ and 
   \[
       (\E D_1)^{-e(H')} \prod_{i=1}^s \E (D_1)_{a_i}  =  \E (D_1)_t  \frac { (\E D_1) ^{s-1}} { (\E D_1)^{s+t-1}} = \beta^{-t} \E (D_1)_t (\E D_2)^{-t}.
   \]
  By comparing $a(H')$ with (\ref{eq.tH}) we see that (\ref{eq.iiH1}) holds. 
\end{proofof}

\bigskip

\noindent
{\bf Acknowledgement} \hspace{.1in} I would like to thank the referees for their helpful remarks, in particular for pointing out the important connections of the definitions used in the paper to the classical theory on the convergence of measures in complete separable metric spaces.

\newpage

\appendix

\section{Proof of Theorem~\ref{thm.riglocal}}
\label{sec.proofs}

In this section we prove Theorem~\ref{thm.riglocal}.
We need different proofs for each of the graph models. The common part in these proofs is that the local weak limit for a sequence of random intersection graphs $\{G_n\}$ that we study is determined by the limit for the sequence of underlying bipartite graphs $\{H_n\}$ and the coupling with a Galton-Watson branching process in each case.

Extending the definition of $p_r(G,H)$ given in Section~\ref{sec.intro}, in the case when $H \in \mathcal{G}_*$ and $G$
is a random rooted graph in $\mathcal{G}_*$, define
$p_r(G, H) = \ii_{B_r(G, \rr(G)) \cong H}$. 
If $G$ is a random rooted or unrooted graph, we will write $\bp_r(G,H) = \E p_r(G,H)$.
Thus $\bp_r(G,H)$ is the probability that an $r$-ball rooted at a random vertex is isomorphic to $H$;
if $G$ is unrooted, then the ball is centered at a uniformly random vertex, while if $G$ is rooted,
then the ball is centered at $\rr(G)$.

For the proofs we need to generalize the notation $p_r, \bp_r$. 
Let $G$ be a random graph with at least $k$ vertices. Let $H_1, \dots, H_k$ be rooted connected graphs. 
We write
\begin{align}\label{eq.eqtrivial}
    &p_r(G, H_1, \dots, H_k) = \E (\ii_{B_r(G, u_1^*) \cong H_1, \dots,  B_r(G, u_k^*) \cong H_k} | G), 
\\  &\bp_r(G, H_1, \dots, H_k) = \E p_r(G, H_1, \dots, H_k).\nonumber
\end{align}
where $(u_1^*, \dots, u_k^*)$ is a uniformly random tuple of $k$ distinct vertices from $V(G)$ (given $V(G)$).

When $G = (V_1, V_2, E)$ is bipartite with $|V_1|, |V_2| \ge k$, 
for $i \in \{1,2\}$ define $p_r^{(i)}$ and $\bp_r^{(i)}$ similarly as $p_r, \bp_r$, but taking a uniformly random
tuple $(u_1^*, \dots, u_k^*)$ of distinct vertices only from the part $V_i$.

Let $\{G_n\}$, $D_1$, $D_2$, $\ct$ be as in Theorem~\ref{thm.riglocal}. Let $H_n$ be the bipartite graph corresponding to $G_n$.
Since $B_r(G_n, v)$ is determined by $B_{2r+1}(H_n,v)$ for $v \in V(G_n)$ and
the subset of graphs in $\mathcal{G}_*$ that have radius at most $2r+1$ is countable,
to prove the theorem it suffices to show in each of the cases that for each positive integer $r$ and each rooted tree $T$
such that $\bp_r(\ct, T) > 0$
\begin{equation}\label{eq.suff}
    p_r^{(1)}(H_n, T) \xrightarrow{p} \bp_r(\ct, T).
\end{equation}
(Use the law of total probability to see that this implies $p_r^{(1)}(H_n, H) \xrightarrow{p} 0$ for any $H$ with $p_r(\ct, H) = 0$.)

We now proceed to prove (\ref{eq.suff}) for each of the four models separately.

\subsection{The active and passive models}

For $i \in \{1,2\}$ define ${\bar i} = 2$ if $i = 1$ and $\bar{i} = 1$ if $i=2$.
Let $(F, S)$ be a pair where
$F$ is a forest of labelled trees, and $S \subseteq V(F)$ a subset of its leaves
(we call a node a leaf if it has degree at most 1). Assume further that $F$ can be represented as a bipartite graph with parts $V_1(F)$ and $V_2(F)$ with $V_i(F) \subset \mathcal{V}^i$. We denote the collection of all such pairs $(F,S)$ by $\ca$.

For a random graph $H_n$ of Theorem~\ref{thm.riglocal} and $(F,S) \in \ca$, we will denote by
$A(H_n, F, S)$ the event that $F$ is a subgraph of $H_n$ and for any $v \in V(F) \setminus  S$ we have $\Gamma_{H_n}(v) = \Gamma_F(v)$.
We will call vertices in $S$ \emph{active} and the vertices in $V(F) \setminus S$ \emph{closed}.

For a discrete random variable $X$, denote by $D(X) = \{x: \pr(X = x) > 0\}$. 
Recall that $X_z = d_{H_n}(z)$ is the degree of $z$ in $H_n$.

\begin{lemma}\label{prop.nocorrelationactive}
    Let $\{G_n\}, \{H_n\}, D_1$ be 
    as in Theorem~\ref{thm.riglocal}(i). 
    Then $H_n$ satisfies (\ref{eq.cond1.2})
    with $D_2 \sim Po(\beta^{-1} D_1)$. 
    
    Furthermore, let $i \in \{1,2\}$ and let $k$ be a non-negative integer.
    Let $(F, S) \in \ca$ and let $z \in S$.
    Call $(F,S)$ feasible if for
    all $u \in (V(F) \cap \cv^1) \setminus S$ we have $d_F(u) \in D(D_1)$. Let $A = A(H_n, F, S)$.
    
    Assume $(F,S)$ is feasible and $z \in S$ is of type $i$ ($z \in \cv^i$). 
    Then for all $n$ large enough $\pr(A) > 0$, 
\begin{align} \label{eq.cond2}
    \lim \pr(X_z = k | A) = 
    \begin{cases}
        \pr(D_i^* = k), \quad \mbox{if }d_F(z) = 1,
        \\ \pr(D_i = k), \quad \mbox{if } d_F(z) = 0,
    \end{cases}
\end{align}     
    and
\begin{align}\label{eq.cond3}
    \lim \pr(\exists u\in V(F): zu \in H_n, zu \not \in F | A) = 0. 
\end{align}
    
\end{lemma}

\begin{proof} 
    Fix a feasible $(F, S)$. 
    Let us first show (\ref{eq.cond1.2}). We use the notation of Section~\ref{sec.intro}. 
    By symmetry $\Xn 2$ has the same distribution as 
    the degree in $H_n$ of any fixed vertex $w$ of $\Vn 2$. By independence of $S_v$, $v \in \Vn 1$, 
    $\Xn 2$ has distribution $Bin(n_1, p)$ where $p = p(n) = \pr(w \in S_{v_1}) =  \frac {\E \Xn 1} {n_2}$ and $v_1$ is a fixed vertex of $\cv^1$. So
    $\E \Xn 2 = \frac {n_1 \E \Xn 1} {n_2}$ and
    $\delta_n = \frac {n_1 \E \Xn 1} {n_2} - \frac {\E D_1} \beta \to 0$ as $n \to \infty$.
    Let $D_2$ and $Y_n$ have Poisson distribution with parameters $\frac {\E D_1} \beta$ and
    $n_1 p$ respectively.
    By Le Cam's theorem and by the properties of Poisson random variables (see, e.g., \cite{bloznelis2011}), we have
    \begin{align}
        &d_{TV}(\Xn 2, D_2) \le d_{TV}(\Xn 2, Y_n ) + d_{TV} (Y_n, D_2) \label{eq.lecam}
       \\ & \le  2 n_1 p^2 + \pr(Po(|\delta_n|) > 0) \to 0. \nonumber
    \end{align}
    Therefore (\ref{eq.cond1.2}) holds with $D_2 \sim Po(\beta^{-1} \E D_1)$. 

    Let $L$ be the set of leaves in $F$.
    We will show the rest of the claim by induction on $|V(F)| + |L \setminus S|$.
    First suppose $F$ is an empty graph and $S=\emptyset$. Then
    trivially $\pr(A(H_n, F,S)) > 0$ and (\ref{eq.cond2}), (\ref{eq.cond3}) hold. Now suppose $l \ge 1$ and we have
    proved the claim for all feasible $(F',S') \in \ca$ such  that $|V(F')| + |L \setminus S| < l$. Let $(F,S) \in \ca$
    be feasible and such that $|V(F)| + |L \setminus S| = l$. 
    
    We begin by proving that
    \begin{equation}\label{eq.Aposit}
        \pr(A(H_n, F, S)) > 0 \quad \mbox{for all $n$ large enough.}
    \end{equation}   
    Indeed, if $F$ has a closed leaf $v$, (\ref{eq.Aposit}) follows by induction and  (\ref{eq.cond2}), (\ref{eq.cond3}) 
    applied to $(F, S \cup \{v\})$. 
    Else, if $F$ has an active vertex $v \in S$ of
    degree zero, then it is trivial. 
    Otherwise all leaves are active.
    Let $v$ be the neighbour of a leaf at a maximum distance from the root. The set $L_{v}$ of the children of $v$ (vertices in $\Gamma_F(v)$ which are not on the path from $v$ to the root) satisfies $L_v \subseteq S$.
    By induction $A(H_n, F - L_v, (S \setminus L_v) \cup \{v\})$ holds with a positive probability.
    Also using induction, (\ref{eq.cond2}), (\ref{eq.cond3}) imply that $v$ has $|L_v|$ neighbours in $G - V(F-L_v)$ with
    a positive probability. Symmetry implies (\ref{eq.Aposit}).  

    In the rest of the proof we assume $n$ is large enough that (\ref{eq.Aposit}) holds and show (\ref{eq.cond2}) and (\ref{eq.cond3}) for $(F,S)$ and an arbitrary $z \in S$.
    Recall that $F$ is a bipartite graph with parts $(V_1(F), V_2(F))$, $V_i(F) \subseteq \cv^i$, $i=1,2$. 
    We assume that $n$ is large enough that $v_i(F)  = |V_i(F)| < n_i$, $i=1,2$. 
    For any $u \in \cv^1$ and $w_1, w_2 \in \cv^2$ we have
    \begin{align}\label{eq.twoedges}
       \pr(u w_1, u w_2 \in H_n) \le \frac {\E \left(\Xn 1\right)^2} {n_2^2} \le \frac {\E \Xn 1 \ii_{\Xn 1 > \sqrt{n_2}}} {n_2} + \frac {\sqrt{n_2} \E \Xn 1} {n_2^2} = o(n_2^{-1}).
    \end{align}
    Here the last bound follows since 
    $\Xn 1$ is uniformly integrable, see e.g., \cite{bloznelis2008} or Lemma~\ref{lem.uint} .

    First consider the case $z \in \mathcal{V}^2$.
    Recall that $S_z= \Gamma_{H_n}(z)$. 
    We may write $X_z  = d_F(z) + X_z' + X_z''$ where $X_z' = |S_z \setminus V(F)|$ 
    and $X_z'' = | S_z \cap (V(F) \setminus \Gamma_F(z))|$.

    Let $\{w_1, \dots, w_r\} =  V_2(F) \setminus S$ be
    the set of closed type 2 vertices in $F$.
    For $u \in \Vn 1 \setminus V(F)$ denote by 
    $B_u = B_u(n)$ the event that  $u w_1, \dots, u w_r \not \in H_n$.
     Write $A = A(H_n, F,S)$.
    By independence
    of $\{S_v, v \in \Vn 1\}$
    \begin{equation}\label{eq.ancond}
        a_n := \pr(uz \in H_n | A) = \pr(uz | B_u)
    \end{equation}
    and note that $a_n$ does not depend on $u \in \Vn 1 \setminus V(F)$.
    Furthermore, we can easily verify that for any distinct $u_1, \dots, u_t \in \Vn 1 \setminus V(F)$
    \[
        \pr(u_1z, \dots, u_t z \in H_n | A) =  \prod_{j=1}^t \pr(z \in S_{u_j} | B_{u_j}) = a_n^t,
    \]
    i.e., the events $uz \in H_n$ for $u \in \Vn 1 \setminus V(F)$ are conditionally independent given $A$.
    Therefore conditionally on $A$, the random variable $X_z'$ is distributed as $X_A'$,
    where $X_A' \sim Bin(n_1 - v_1(F), a_n)$.

Let us estimate $a_n$. For any  $u \in \Vn 1 \setminus V(F)$ 
we have using (\ref{eq.twoedges}) and the union bound
    \begin{align}
        &\pr(B_u) \ge 1 - r\pr(uw_1 \in H_n) \ge 1 - \frac {r \E \Xn 1} {n_2} = 1 - o(1); \nonumber 
     \\  &\pr(uz \in H_n, B_u) \le  \pr(uz \in H_n) = \frac {\E \Xn 1} {n_2}; \nonumber
     \\ &\pr(uz \in H_n, B_u) \ge \pr(uz \in H_n) - r \pr(uz, uw_1 \in H_n) = \frac {\E \Xn 1} {n_2} \left( 1 - o(1) \right).\label{eq.onekeyandrubbish}
    \end{align}
    So by (\ref{eq.cond1.1}) and (\ref{eq.ancond})
    \[
        a_n =  \frac {\E \Xn 1} {n_2} \left( 1 - o(1) \right) =  \frac {\E D_1} {n_2} \left( 1 - o(1) \right).
    \]
    Suppose $Z_n \sim Bin(n_1,a_n)$. Similarly as in (\ref{eq.lecam})
    \begin{equation}\label{eq.dtv1}
        d_{TV}(X_A', D_2) \le d_{TV} (X_A', Z_n) + d_{TV}(Z_n, D_2) \le v_1(F) a_n + o(1) \to 0.
    \end{equation}
    Now consider $X_z''$. 
    Suppose $u \in (S \cap \Vn 1) \setminus \Gamma_F(z)$. If $d_F(u) = 1$, we may assume without loss of generality that
    the neighbour of $u$ in $F$ is $u'$. 
    Write $W' =  \{w_1, \dots, w_r\} \setminus \{u'\}$.
    Using independence, (\ref{eq.twoedges}) and (\ref{eq.onekeyandrubbish})
    \begin{align*}
        &\pr(uz \in H_n | A) = \pr(uz \in H_n | u u' \in H_n, \cap_{w \in W'} u w \not \in H_n) 
        \le \frac {\pr(u z, u u' \in H_n)} {\pr(\cap_{w \in W'} u w \not \in H_n)} = o(1) .
    \end{align*}
    Similarly, if $d_F(u) = 0$,
    \begin{align*}
        &\pr(uz \in H_n | A) \le \frac {\pr(uz \in H_n)} {\pr(uw_1, \dots, uw_r \not \in H_n)} = \frac {\E \Xn 1} {n_2} (1 + o(1)) = o(1)
    \end{align*}
    Thus by the union bound $\pr(X_z'' > 0 | A)  = o(1)$: this yields (\ref{eq.cond3}) when $z$ is of type 2. Using (\ref{eq.dtv1}) we
    conclude for $k \ge 1$ and $z \in S \cap \cv^2$
    \[
        \pr(X_z = k | A)  \to  \begin{cases} 
            &\pr(Po(\E D_1 /\beta) = k), \quad \mbox{if } d_F(z) = 0; \\
            &\pr(Po(\E D_1 /\beta) + 1 = k), \quad \mbox{if } d_F(z) = 1.
         \end{cases}
    \]
    Since $\Po(\lambda)^* \sim 1 + \Po(\lambda)$ (see Lemma~\ref{lem.cpoissonbias}), (\ref{eq.cond2}) follows when 
    $z \in S \cap \cv^2$.
 
    \bigskip

    Now suppose $z \in \mathcal{V}^1$. We will prove the claim only in the case
    $d_F(z) = 1$. The case $d_F(z) = 0$ is similar, but simpler.
    Let $z'$ be the only neighbour of $z$ in $F$, and let
    $\{w_1, \dots, w_r\} =  V_2(F) \setminus (S \cup \{z'\})$. Write 
    $B = B(n)$ for the event that  $zw_1, \dots, zw_r \not \in H_n$. We have
    \[
        \pr(X_z = k | A) = \pr(X_z = k | zz' \in H_n, B)
    \]
    and using symmetry
    \begin{align*}
        &\pr(X_z = k, zz' \in H_n, B) = P(zz' \in H_n| B, X_z = k) \pr(X_z = k, B)
        \\      &= \frac k {n_2-r} (\pr(D_1 =k) + o(1)).
    \end{align*}
    The estimate $\pr(X_z = k, B) = \pr(D_1 =k) + o(1)$ follows since by the union bound and (\ref{eq.twoedges}), 
    $0 \le \pr(X_z=k) - \pr(X_z = k, B) \le r \pr(w_1 \in S_z) = O(n_2^{-1})$
    and
$\pr(X_z = k) \to \pr(D_1 = k)$ by (\ref{eq.cond1.1}). Therefore using also (\ref{eq.onekeyandrubbish})
    \[
        \pr(X_z = k | A) =  \frac {\pr(X_z = k, zz' \in H_n, B)}{\pr(zz' \in H_n, B)} \to \frac {k \pr(D_1 = k)} {\E D_1} = \pr(D_1^* = k). 
    \]
    Finally, we show (\ref{eq.cond3}), also only the case $d_F(z) = 1$.
    This is trivial if $s = |S \cap \cv^2| = 0$.
    If $s \ge 1$, fixing any $w_0 \in \mathcal{V}^2 \setminus \{z'\}$ and using
    (\ref{eq.twoedges}) and (\ref{eq.onekeyandrubbish})
    \begin{align*}
        &\pr( |S_z \cap V(F)| \ge 2 | A) \le \sum_{w \in (S \cap \Vn 2) \setminus \{z'\}} \frac {\pr(w z, z z' \in H_n, B)} {\pr(z z' \in H_n, B)}
       \\ &\le \frac {s n_2 \pr(z w_0, zz' \in H_n)} {\E \Xn 2} (1 + o(1)) = o(1).
    \end{align*}
    \end{proof}

The next lemma follows using a lengthy but essentially trivial argument.

\begin{lemma}\label{lem.trivial}
    Let $\{H_n\}$ be as in Theorem~\ref{thm.riglocal}(i). 
    Fix $i \in \{1,2\}$, a non-negative integer $r$ and let $C_1$, \dots, $C_k$ be 
    rooted trees 
    of radius at most $r$.
    Let $\ct^{(i)} = \ct(D_i, D_{\bar i})$ where $D_1, D_2$ are given in (\ref{eq.cond1.1}), (\ref{eq.cond1.2}). 

    Suppose that $\prod_{j=1}^k \bp_r(\ct^{(i)}, C_j) > 0$.
    Then
    \begin{align}\label{eq.eqeq}
        \bp_r^{(i)}(H_n, C_1, \dots, C_k) \to \prod_{j=1}^k \bp_r(\ct^{(i)}, C_j)
    \end{align}
\end{lemma}
\begin{proof}
The lemma is trivial for $r=0$, so assume $r \ge 1$. 
We couple a BFS search process with the construction of the first $r$ generations of
    $k$ independent copies of $\ct^{(i)}$ in the natural way as follows.
 Let $F' = C_1' \cup \dots \cup C_k'$ be an arbitrary embedding of $F = C_1 \cup \dots \cup C_k$ into plane. 
Let $v_1, v_2, \dots, v_t$ be the vertices of $F'$ in the order they are visited by the
breadth-first search that terminates at level $r-1$ and let $d_j = d_{F'}(v_j)$. (The BFS starts with the root $v_1$ of $C_1'$ and
visits each vertex in $B_{r-1}(v_1)$. After exploring $C_j'$ with $j < k$
it jumps to the root of $C_{j+1}'$. The degrees of each tree are listed in the top-to-bottom, left-to-right order,
where the root of a tree is at the top, and the children of each node are ordered from left to right).

To this exploration of a fixed forest associate a ``truncated'' BFS exploration of
the graph $H_n$.

At step $1$ we choose as a root of $G_1$ a uniformly random vertex $u_1^*$ from $\Vn i$
and reveal its neighbours $S_{u_1^*}$ in $H_n$. Define $V^*_1 = \{u_1^*\}$ and $\phi(v_1) = u_1^*$. 
We say that step 1 \emph{succeeds} in $H_n$ if $X_{u_1^*}^1 = d_1$. 

Now let $j \in \{2,3,\dots,t\}$, and assume the steps $1,\dots,j-1$ succeeded. 
At step $j$ we do the following.
If $v_j$ is not a root node in $F'$ then let $\tilde{v}$ be 
the parent of $v_j$ in $F'$. Let $u = \phi(\tilde{v})$.
Choose $x$ uniformly at random
from $S_u \setminus V_{j-1}^*$.
Reveal  $d_j' := X_x$, the degree of $x$ in $H_n$ and
the remaining $X_x - 1$ elements of the set $S(x)=\Gamma_{H_n}(x)$ and say
that step $i$ succeeds if $d_j' = d_j$
and no vertex in $S(x) \setminus \{u\}$ was revealed before step $j$.
Finally set $V_j^* = V^*_{j-1} \cup \{x\}$, $\phi(v_j) = x$.

If $v_j$ is a root node, then we let $x$ be a uniformly random
vertex from $\Vn i \setminus R_{j-1}$, where $R_{j-1} = \{u \in V(H_n): p \le j-1, \phi(v_p) = u, v_p \in R\}$ and
where $R$ is the set of the root nodes of $F'$.
We reveal $S(x)$, set $\phi(v_j)=x$ and define $V^*_j =  V_{j-1}^* \cup \{x\}$.
We say that step $j$ succeeds if $d_j':= X_x$ is equal to $d_j$ and no vertex in $S(x) \cup \{x\}$ 
was revealed in previous steps.

Finally, in the third process we simply generate the first $r$ generations
    of $k$ independent copies of the branching process $\ct^{(i)}$.
We query the random number of children for each node in the same order as the BFS search process on $F$. 
If the degree $d_j''$ of a particle queried at step $j$ is equal to $d_j$, we say that the step $j$ in the branching process succeeds.


Let $S_j$ be the event that step $j$ succeeds in $H_n$. Let $U^*_{j-1}$ be the set of vertices revealed until step $j$, i.e.,
$U_{j-1}^* = V_{j-1}^* \cup \{u: \exists v \in V_{j-1}^*: \, uv \in H_n\}$. 

Let $F_{j-1}'$ be the subgraph of $F'$ induced on vertices $\{v_1, \dots, v_{j-1}\}$ and their neighbours in $F'$.
Let $F_{j-1}$ be an arbitrary fixed bipartite graph
with parts $V_1(F_{j-1}) \subset \mathcal{V}^i$ and $V_2(F_{j-1}) \subset \mathcal{V}^{\bar i}$ such that $F_{j-1} \sim F'_{j-1}$. Recall that two bipartite graphs satisfy $H_1 \sim H_2$ if and only if there is a part-preserving isomorphism between $H_1$ and $H_2$; fix such an isomorphism $\sigma_j$ from $F_{j-1}$ to $F'_{j-1}$.
Also define $T_{j-1} := V(F'_{j-1}) \setminus \{\sigma_j(v): v \in \{v_1, \dots, v_{j-1}\} \}$

First suppose $v_j$ is a root node of $F'$. Since $v(F')=|V(F')|$ is constant, using
an inequality $|\pr(B) - \pr(B|C)| \le \pr({\bar C})$ valid for arbitrary events $B, C$ with $\pr(C) > 0$
\begin{align*}
    &\left | \pr(S_j | S_1, \dots, S_{j-1}) - \pr(S_j | \phi(v_j) \not \in U_{j-1}^*, S_1, \dots, S_{j-1}) \right|
    \\ &\le \pr(\phi(v_j) \in U_{j-1}^* | S_1, \dots, S_{j-1}) \le \frac {v(F')} {n_1} = o(1).
\end{align*}
Fix $z \in \mathcal{V}^i \setminus V(F_{j-1})$, and let $F_{j-1}^+$ be the graph
consisting of $F_{j-1}$ and an isolated vertex $z$.
Let $C_j$ be the event that $d_{H_n}(z) = d_j$ and there is no edge from $z$ to $V(F_{j-1})$ in $H_n$.  

Let $A_j = A(H_n, F_{j-1}^+, T_{j-1} \cup \{z\})$. 
Conditioning on $\phi(v_1), \dots, \phi(v_j)$, using symmetry
and Lemma~\ref{prop.nocorrelationactive} we get
\begin{align*}
    \pr(S_j | \phi(v_j) \not \in U_{j-1}^*, S_1, \dots, S_{j-1}) = \pr(C_j | A_j) \to \pr(D_i = d_j).
\end{align*}
Thus $a_j = \pr(S_j | S_1, \dots, S_{j-1}) \to \pr(D_i = d_j)$.

Now suppose $v_j$ is a non-root node of $F'$.
Let $A_j = A(H_n, F_{j-1}, T_{j-1})$ and let $z \in V(F_{j-1})$ be such that $\sigma_j(z) = v_j$,
and suppose $z \in \mathcal{V}^p$. 
Then using symmetry,  Lemma~\ref{prop.nocorrelationactive} similarly as above
\begin{align*}
    a_j = \pr(S_j | S_1, \dots, S_{j-1}) = \pr(C_j | A_j) \to \pr(D_p^* = d_j).
\end{align*}



Note that $S_1 \cap \dots \cap S_t$ implies that there is an isomorphism
from $B_r(H_n, u_1^*) \cup \dots \cup B_r(H_n, u_k^*)$ 
to $C_1' \cup \dots C_k'$ mapping the root of $B_r(H_n, u_j^*)$ to the root
of $C_j'$ and preserving the order of the plane embedding. This follows by our coupling
of the explorations and the
definition of the events $S_j$ and $A_j$: all edges incident to 
$\{\phi(v_1), \dots, \phi(v_t)\}$ in $H_n$
are specified by $\cap S_j$; for $j=t$ we have that all
vertices of $T_{t-1} \setminus \{\sigma_t(v_j)\}$ are at distance $r$ from the root, but $H_n$ 
is bipartite, so there cannot be an edge in $H_n$ between them.

Since the number of vertices in $F$ is constant, we get
$\pr(S_1, \dots, S_t) \to \prod_j a_j$, which is exactly the probability for all steps in the branching process to succeed.
Since this holds for each embedding $F'$ of $F$, (\ref{eq.eqeq}) follows.

\end{proof}

\bigskip

%


\bigskip

    \begin{proofof} {Theorem~\ref{thm.riglocal} (i), (ii)}
        \label{pf.riglocal.i-ii} 
        We first show (i). 
        Let $\ct = \ct(D_1, D_2)$.
        Lemma~\ref{lem.trivial} implies that for
        any set of rooted trees  $C_1, \dots, C_k$
        \begin{equation}\label{eq.bpg}
            \E p_r(G_n, C_1,\dots,C_k) = \bp_r(G_n, C_1, \dots, C_k) \to \prod_{j=1}^k \bp_r(G_{\ct}, C_j).
        \end{equation}
        For $v \in \Vn 1$ let $\ii_v$ be the indicator of the event $B_r(G_n, v) \cong C$. Denote $p = p(n) = \bp_r(G_n, C_1)$,
        $q = q(n) = \bp_r(G_n, C_1, C_1)$. By (\ref{eq.bpg}), $q - p^2 \to 0$, 
        \begin{align*}
            &Var(p_r(G_n, C_1)) = Var\left( n_1^{-1} {\sum_{v\in \Vn 1} \ii_v}\right) = n_1^{-2} \sum_{u,v \in \Vn 1} Cov(\ii_u, \ii_v)
            \\ &= n_1^{-1} (p - p^2) + \frac {n_1 (n_1-1)} {n_1^2} (q - p^2) \to 0. 
        \end{align*}
        Since $\E p_r(G_n, C_1) \to \bp_r(G_\ct, C_1)$, we have $p_r(G_n, C_1) \xrightarrow{p} \bp_r(G_\ct, C_1)$
        by Chebyshev's inequality. Now (i) follows by (\ref{eq.suff})..

        The proof of (ii) follows analogously, but using $i=2$ in Lemma~\ref{lem.trivial}. 
        Indeed, if $H_n = (\Vn 1, \Vn 2, F)$ and $G(H_n)$ is a passive random intersection graph with $n_i = |\Vn i|$, then
        $G(H_n')$ with $H_n' = (\Vn 2, \Vn 1, F)$ is an active random intersection graph with parts of size $n_1' = n_2$ and $n_2' = n_1$.
        In particular, in this case $\beta' = \lim \frac {n_2'} {n_1'} = \beta^{-1}$ and by Lemma~\ref{prop.nocorrelationactive} (\ref{eq.cond1.1}) holds with $D_1 \sim Po(D_2/\beta') \sim Po(\beta D_2)$.
    \end{proofof}

\subsection{The inhomogeneous model}
\label{sec.inhomog}

We will reduce Theorem~\ref{thm.riglocal}(iii) to the general model of Bollob\'as, Janson and Riordan \cite{bollobasjansonriordan2007, bollobasjansonriordan2011}. 
Results of \cite{bollobasjansonriordan2007} have also been applied by Bloznelis \cite{bloznelis2010} to study the largest connected component in
an inhomogeneous random intersection graph. 




Let $\mathbb{S} = (S, \mu)$ be a probability space and let $\kappa = \kappa_{K_2}: S^2 \to [0,\infty)$ be a measurable function, or a \emph{kernel}. The random
inhomogeneous graph $G(n,\kappa)$ on the vertex set $\{v_1, \dots, v_n\}$ is obtained by sampling independently at random 
points $x_1, x_2, \dots, x_n$ from the distribution $\mu$. Then for each ordered\footnote{For our simple application it would be more
    natural to add each unordered pair with probability $2p_{i,j}$. The results are equivalent, but for the sake of consistency we stick to the definition of \cite{bollobasjansonriordan2011}.} pair $(v_i, v_j)$ we add an edge $v_i v_j$ to $G(n, \kappa)$ with probability $p_{i,j} = \min (\frac {\kappa(x_i, x_j)} n, 1)$, independently, merging
    any repetitive edges. For $i \in \{1, \dots, n\}$ we call $x_{v_i} = x_i$ the \emph{type}\footnote{In this section we use the notion \emph{type} to refer to an element from $S$, not to the part of a vertex in a bipartite graph.} of $v_i$.

Assuming that $\kappa$ and is integrable, the associated Galton-Watson branching process $\mathcal{X}_\kappa$ is defined as follows \cite{bollobasjansonriordan2011}.
There is a single particle  in generation 0. The type of this particle is chosen from $S$ according to the distribution $\mu$. The children of each particle $P$ of type $x$ have types which are the points of
a Poisson process with intensity $2 \kappa(x, y) d \mu(y)$ ($f$ defined by $f(x) = \int \kappa(x,y) d\mu(y)$ is finite everywhere except a set of measure 0 in $S$ which we may ignore). The children for each particle in the same generation
are generated independently, and all the children of generation $i$ particles make up generation $i+1$.
We use the same notation $\mathcal{X}_\kappa$ to denote the corresponding random possibly infinite rooted tree.

Although the work \cite{bollobasjansonriordan2007} probably contains what is necessary for our proof,
it is simpler in our case to use the results of \cite{bollobasjansonriordan2011}. 
These results also hold in a more general framework where kernels for arbitrary small subgraphs (for example, cliques) are allowed. However, a single kernel function (corresponding to the subgraph $K_2$) is sufficient to model the underlying bipartite graph $H_n$ of the inhomogeneous random intersection graph.

\begin{lemma}{(Theorem 9.1 of \cite{bollobasjansonriordan2011})} \label{lem.bjrlimit}
    Suppose $\kappa$ is integrable on $S^2$. Then for any non-negative integer $r$ and any rooted connected $H$ 
    \[
        p_r(G(n, \kappa), H) \xrightarrow{p} \bp_r(\mathcal{X}_\kappa, H).
    \]
\end{lemma}
That is, $\cl((G(n, \kappa),v^*_n)|G_n) \weaklypto \cl(\mathcal{X}_{\kappa})$ for $v_n^* \in_u \{v_1,\dots, v_n\}$ in our notation above.

The proof of \cite{bollobasjansonriordan2011} is elegant and based on approximation using a bounded kernel and embedding
the inhomogeneous branching process into a homogeneous process. 
To apply the result for $G(n, \kappa)$ to random bipartite graphs, we need a minor technical
modification of the above lemma (we omit the proof).
For a graph $G$, a set $A \subseteq V(G)$ and a rooted connected graph $H$, let $p_r(G, H, A)$ denote the probability that $B_r(G, u^*) \cong H$, where $u^* \in_u A$ (define $p_r(G, H, A) :=0$ if $A$ is empty). 
\begin{lemma}\label{lem.probconv}
Suppose $\kappa$ is integrable on $S^2$. Then for any non-negative integer $r$, any rooted connected $H$ 
and any measurable set $A \subseteq S$ such that $\mu(A) > 0$, if $\tilde A$
denotes the set of vertices $v$ of $G(n, \kappa)$ such that $x_v \in A$ then
\begin{equation}\label{eq.localconv}
    p_r(G(n, \kappa), H, \tilde{A}) \xrightarrow{p} \bp_r(\mathcal{X}_{\kappa, A}, H).
\end{equation}
Here $\mathcal{X}_{\kappa, A}$ denotes $\mathcal{X}_\kappa$ conditioned on the event that the type of the root is in $A$.
\end{lemma}



We also need a simple fact about size-biased mixed Poisson distributions.
\begin{lemma}\label{lem.cpoissonbias}
    Let $\Lambda$ be a non-negative random variable with $0 < \E \Lambda <\infty$.
    Suppose $X \sim Po(\Lambda)$. Then the corresponding size-biased random variable satisfies
    \[
        X^* \sim Po(\Lambda^*) + 1
    \]
\end{lemma}
\begin{proof}
    Let $\phi_\Lambda$ be the characteristic function of $\Lambda$. The ch. f. of $X$ is
    \[
        \phi_X(t) = \E e^{it X} = \phi_\Lambda(-i(1 - e^{it})).
    \]
    Also, $\E X = \E \Lambda$.
    If a random variable $Z$ with $0 < \E Z < \infty$ has a ch. f. $\phi_Z$, then the
    ch. f. of $Z^*$ is $(i \E Z)^{-1} \phi_Z'(t)$, see, e.g., \cite{arratiagoldstein2009}.
    We see that the ch. f. of $X^*$,
    \[
        \phi_{X^*}(t) = (i \E \Lambda)^{-1} e^{it} \phi_\Lambda'(-i(1-e^{it}))
    \]
    is equal to the ch. f. of $Y \sim Po(\Lambda^*) +1$:
    \[
        \phi_Y(t) = e^{it} \phi_{\Lambda^*}(-i(1 - e^{it})) = (i \E \Lambda)^{-1} e^{it} \phi_{\Lambda}'(-i(1 - e^{it})).
    \]
\end{proof}

\bigskip


    Let $\beta > 0$, $\xi^{(1)}, \xi^{(2)}, n_1, n_2$ be as in Theorem~\ref{thm.riglocal}(iii). We assume $\xi^{(1)}, \xi^{(2)}$ are independent. For $i = 1,2$ define
    $\tilde{\xi}^{(i)} = (1+\beta)^{1/2} \beta^{-1/4} \xi^{(i)}$. 

    Consider a random inhomogeneous graph $G(N, \kappa)$ where $\mathbb{S} = (S, \mu)$ and $\kappa = \kappa_\beta: S^2 \to [0,1]$ are as follows.
    Let $S = \{1,2\} \times \mathbb{R}$, and let $\mu$ be a measure induced by the random vector 
    \begin{equation} \label{eq.Sdef}
        X = (i_X, w_X) = (1+\ii, (1-\ii) \tilde{\xi}_1 + \ii \tilde{\xi}_2)
    \end{equation}
    where $\ii$ is a Bernoulli random variable with parameter $\beta (1+\beta)^{-1}$ independent of $\{\xi_1, \xi_2\}$. Finally,
    set for $x, y \in S$
    \begin{equation}\label{eq.kappadef}
        \kappa(x,y) = \begin{cases}
                        \frac 1 2 w_x w_y, &\mbox{ if } i_x \ne i_y; \\
                        0, &\mbox{ otherwise }.
                   \end{cases}
    \end{equation}
    We consider the graph $G(N, \kappa)$ as a bipartite graph $(V^{1}, V^{2}, E)$, with $V^{i}$ consisting of all vertices $v_j$ such that $i_{x_j} = i$.

\begin{prop} \label{prop.inhomogbranching}
    Let $\mathcal{X}_\kappa$ be the branching process corresponding to $G(N,\kappa)$ defined above. Let 
    $X=(i_X,w_X)$ be the random type of the root of $\mathcal{X}_\kappa$. 
    Fix $i \in \{1, 2\}$.
    Let $\ct^{(i)} \sim \ct(D_i, D_{\bar i})$, where 
\[
    D_i \sim Po(c_i \xi^{(i)} \E \xi^{(\bar{i})}), \quad c_1 = \beta^{1/2}, c_2 = \beta^{-1/2}.
\]
Then for any 
    non-negative integer~$r$ and a rooted tree $T$ of radius at most $r$
 \[
     \pr(B_r(\cx_\kappa) \cong T | i_X = i) = \bp_r(\ct^{(i)},T). 
 \]
\end{prop}

\begin{proof}
 Recall that $X$ has unconditional distribution $\mu$ defined in (\ref{eq.Sdef}). Let $X_2$
 be an independent copy of $X$.

     We have $\pr(i_X = i) = q_i$, where $q_1 =(1+\beta)^{-1}$ and $q_2 = \beta (1+\beta)^{-1}$. 
Write $s= {\bar i}$. 
By definition, conditioned on $X = x = (i, w)$, the types of children of the root are  points of a Poisson process with intensity
\[
    \mu_{i,w}(z)=2 \kappa(x, z) d \mu(z) = w w_{z} \ii_{i \ne i_{z}} d \mu(z).
\]
Thus, given $X= x= (i,w)$, 
the number $d_{0,1}$ of children of the root
is distributed as
\begin{equation}\label{eq.porootchildren}
    Po(w \E \ii_{i_{X_2} \ne i} \tilde{\xi}^{(s)}) = Po(w q_{s} \E \tilde{\xi}^{(s)}).
\end{equation}
Given $X= (i, w)$, $d_{0,1} = k$, the types of the children of the root are independent elements
$X_{1, 1}, \dots, X_{1,k}$ of $S$, where  $X_{1,j} = (s, W_{1,j})$ and $W_{1,j}$
has distribution $\mu_s$ given by 
\begin{equation}\label{eq.muinh}
    \mu_s(A) = \frac {\int_{A} w_x q_{s} t d \mu_{\tilde{\xi}^{(s)}}(t)} {w_x q_{s} \E \tilde{\xi}^{(s)}} 
    = (\E \tilde{\xi}^{(s)})^{-1} \int_A t d \mu_{\tilde{\xi}^{(s)}}(t)
\end{equation}
for each Borel set $A$. Here we used the well known property on the distribution
of the points in an inhomogeneous random process given that the number of points is $k$, see, e.g. \cite{durrett}.
Note that $\mu_s$ does not depend on $w$ and it is the distribution of the size-biased random variable 
\[
    \left(\tilde{\xi}^{(s)}\right)^* = (1+\beta)^{1/2} \beta^{-1/4} \left(\xi^{({s})}\right)^*.
\]
Thus, using (\ref{eq.porootchildren}) and the fact that $w_X \sim \tilde{\xi}^{(i)}$ is independent of $i_X$, 
for any $k = 0,1, \dots$
\begin{align}
    \pr(d_{0,1} = k | i_X = i) &= \frac{\E \E (\ii_{d_{0,1}= k} \ii_{i_X = i} | w_X)} {\E \ii_{i_X = i}} \nonumber
    \\ &
    =\pr (Po(w_X q_{\bar i} \E \tilde{\xi}^{({\bar i})}) = k) = \pr(D_i = k). \label{eq.inhpoisdeg}
\end{align}
Also, by (\ref{eq.muinh}) for any $k = 0, 1, \dots$ and  $(t_1, \dots, t_k) \in \mathbb{R}^k$
\begin{align}
    &\pr(W_{1,1} \le t_1, \dots, W_{1,k} \le t_k | i_X = i, d_{0,1} = k)  \nonumber
    \\ &
    =
    \frac{\E \E \left( \ii_{W_{1,1} \le t_1, \dots, W_{1,k} \le t_k} \ii_{i_X=i} \ii_{d_{0,1}=k} | w_X \right)} 
    {\E \ii_{i_X=i} \ii_{d_{0,1}=k} }
    = \prod_{j=1}^k \pr(\tilde{\xi}_j^{({\bar i})} \le t_j). \label{eq.dstarprod}
\end{align}

(Here, to formally establish the connection of (\ref{eq.porootchildren}) and (\ref{eq.muinh}) with the conditional expectations
in (\ref{eq.inhpoisdeg}) and (\ref{eq.dstarprod}) respectively, construct
the probability space for $\cx_\kappa$ as a countable product of spaces generated
by $w_X$ and a sequence $Z$ of other independent random variables, see, e.g., Example~3.6.8 of \cite{durrett}.
Then define  a measurable function $\phi$,
so that $g(w) = \phi(w, Z)$ is the random variable of interest when we consider $w_X=w$ fixed.
Here the random variables of interest are of the form $\ii_{d_{0,1} = k} \ii_{i_X = i}$ and 
$\ii_{W_{1,1} \le t_1} \dots \ii_{W_{1,k}\le t_k} \ii_{d_{0,1} = k} \ii_{i_X = i}$.
Finally, use Radon-Nikodym's theorem, see Example~5.1.5 of \cite{durrett}.)


The trees $\ct^{(i)}$ and $\cx_\kappa$ can both be considered as random rooted plane trees,
i.e., rooted trees where the offspring of each node is ordered. Below
we consider the set of generation $r$ vertices of 
a plane tree $T$ as ordered (e.g., by the discovery times of the breadth-first search in $T$).
For a plane tree $T$, we let $\tilde{B}_r(T)$ denote the plane tree induced on the generations $0, \dots, r$
of $T$. For two rooted plane trees $T_1, T_2$ we write $T_1 = T_2$ if there is
an isomorphism from $T_1$ to $T_2$ preserving the root and the ordering for each vertex.

Let $S(r)$ be the sequence of vertices at distance $r$ from the root in $\cx_\kappa$.
For $r =  1,2,\dots,$ let $i_r = \bar{i}$ if $r$ is odd and $i_r = i$ if $r$ is even.
To complete the proof, we show the following claim by induction on $r$:
for any finite rooted plane tree $T$ of radius at most $r$, and any $i=1,2$,
$\pr(\tilde{B}_r(\cx_\kappa) = T | i_X = i) = \pr(\tilde{B}_r(\ct^{(i)}) = T)$; furthermore given $i_X = i$ and $B_r(\cx_\kappa) = T$, the types of vertices in $S(r)$  are conditionally independent  
and distributed as 
\begin{equation}\label{eq.claimind}
    \left(i_r, \left(\tilde{\xi}^{(i_r)}\right)^*\right).
\end{equation}
Indeed, we have already shown the claim for $r=1$. Suppose $r \ge 2$.
Let $T$ be a rooted plane tree of radius at most $r$. If the radius of $T$ is less than $r$, we are done.
So assume there are $m \ge 1$ and $q \ge 1$ vertices in generations $r$ and $r-1$ of $T$ respectively. Let $T_{r-1} = \tilde{B}_{r-1}(T)$.
 Denote by $X_{a, j} = (I_{a, j}, W_{a,j})$ the type and by
$d_{a, j}$ the number of offspring of the $j$-th vertex in $S(a)$.
Of course, for any $j=1, \dots, q$
$\pr(I_{r-1, j} = i_{r-1}| \tilde{B}_{r-1}(\cx_\kappa) = T_{r-1}, i_X=i) = 1$ by induction.
By (\ref{eq.porootchildren}) and induction,  the conditional distribution of  $(d_{r-1,1}, \dots, d_{r-1, q})$
given $i_X=i$ and $\tilde{B}_{r-1}(\cx_\kappa)=T_{r-1}$, is the same as the distribution
of a vector of $q$ iid random variables $(\eta_1, \dots, \eta_q)$,
\[
   \eta_j \sim Po\left(W_{r-1,j} q_{i_r} \E \xi^{(i_r)}\right) \sim Po\left( \left({\tilde \xi}^{(i_{r-1})}\right)^* q_{i_r} \E \xi^{(i_r)}\right) \sim D_{i_r}^* - 1.
\] 
Let $y_j$ be
the number of children in $T$ of the $j$-th vertex in generation $r-1$ of $T$. We have
\begin{align}
    &\pr(d_{r-1, 1} = y_1, \dots, d_{r-1, q} = y_q | i_X = i, \tilde{B}_{r-1}(\cx_\kappa) = T_{r-1}) \nonumber
    \\ &=\frac {\E \E \ii_{i_X = i} \ii_{ \tilde{B}_{r-1}(\cx_\kappa) = T_{r-1}} \prod_{j=1}^q \ii_{\eta_j = y_j}} { \E \ii_{i_X = i} \ii_{ \tilde{B}_{r-1}(\cx_\kappa) = T_{r-1}}} 
    =\prod_{j=1}^q \pr(D_{i_{r-1}}^* = y_j). \label{eq.dstarnn}
\end{align}
By (\ref{eq.dstarnn}) and induction
\begin{align*}
    \pr({\tilde B}_r(\cx_\kappa) = T | i_X = i) &= \pr({\tilde B}_{r-1}(\cx_\kappa) = T_{r-1} | i_X = i) \times 
    \\
      &\pr(d_{r-1, 1} = y_1, \dots, d_{r-1, j} = y_q | i_X = i, \tilde{B}_{r-1}(\cx_\kappa) = T_{r-1}) 
    \\  &= \pr({\tilde B}_{r-1}(\ct^{(i)}) = T_{r-1}) \prod_{j=1}^q \pr(D_{i_{r-1}}^* = y_j) = \pr({\tilde B}_{r}(\ct^{(i)}) = T)
\end{align*}
Using (\ref{eq.muinh}), we see that given  $i_X = i$, $B_{r-1}(\cx_\kappa) = T_{r-1}$, $X_{r-1,j} =(i_{r-1}, w_{r-1,j})$ and
$d_{r-1,j} = y_j$, $j=1, \dots, q$, 
the types of offspring of each vertices in $S(r-1)$
are independent and distributed as $(i_r, \left(\tilde{\xi}^{(i_{r})}\right)^*)$.
By the conditional independence of the offspring of $S(r-1)$
we get, integrating over $w_{r-1,1}, \dots, w_{r-1, q}$, similarly as in (\ref{eq.dstarprod}), that for any
$t_1, \dots, t_m \in \mathbb{R}$
\[
    \pr(W_{r,1} \le t_1, \dots, W_{r,m} \le t_m | i_X = i, T_r = T) = \prod_{j=1}^m \pr((\tilde{\xi}^{(i_r)})^* \le t_j).
\]
The measure theoretic details concerning conditional expectations here may be filled using a similar argument as above.
\end{proof}

\bigskip

We will need the following fact based on ideas from \cite{bollobasjansonriordan2011}. The statement and the proof
would also hold more generally when $\kappa$ is an ``integrable kernel family'', see \cite{bollobasjansonriordan2011}. Recall that $S \triangle S'$
denotes the symmetric difference between sets $S$ and $S'$. We will apply the lemma with small $t$. 
\begin{lemma} \label{lem.smalldiff}
    Fix a rooted connected graph $H$, a non-negative integer $r$ and $t > 0$. 
    Let $G(n, \kappa), A$ and ${\tilde A}$ be as in Lemma~\ref{lem.probconv}.
    Let $G'_n$ be another random graph and $A'$ a random subset of $V(G'_n)$. 

    Write $\Delta_V = V(G(n,\kappa)) \triangle V(G'_n)$, $\Delta_E = E(G(n,\kappa)) \triangle E(G'_n)$,
    $\Delta_A = {\tilde A} \triangle A'$.
    There is $\delta > 0$ such that if $\pr(|\Delta_V| + |\Delta_E| + |\Delta_A| \le \delta n) \to 1$, then
    \[
        \pr(|p_r(G'(n), H, A') - c_H| > t) \to 0, 
    \]
    where $c_H = \bp_r(\mathcal{X}_{\kappa, A}, H)$.
\end{lemma}

\begin{proof}
    It suffices to prove the lemma for $c_H > 0$, the proof for $c_H=0$ follows by Lemma~\ref{lem.bjrlimit}, Lemma~\ref{lem.probconv}
    and the law of total probability, in the same way as (\ref{eq.suff}).
    We may assume $t < 0.1$. Write $a = \mu(A)$ and fix $\eps \in (0,  \frac 1 8 a c_H t)$. 

    Call a vertex $v$ of $G_n = G(n,\kappa)$ \emph{bad} if it is in $\Delta_V \cup \Delta_A$ or it is incident to an edge in $\Delta_E$.
    Let $B$ the set of all bad vertices.
    Then $|B| \le  |\Delta_V|  + 2|\Delta_E| + |\Delta_A|$.

    Lemma~9.3 of \cite{bollobasjansonriordan2011} states that for each $\eps' > 0$ we can find
    $\delta' > 0$ such that for each set $Q$ of size at most $\delta' n$ the number of vertices in $G_n$ at distance at most 1 from
    $Q$ is at most $\eps' n$ with probability $1-o(1)$. 

    By $r$ applications of this lemma if follows that if $\delta$ is sufficiently small and 
     $|\Delta_V|  + |\Delta_E| + |\Delta_A| \le \delta n$ then
    the set $S$ of vertices $v \in V(G(n, \kappa))$ at distance at most $r$ in $G(n,\kappa)$ from $B$
    satisfies $|S| \le \eps n$ with probability $1-o(1)$. We can assume that $\delta \le \eps$.

    Call a realisation of $(G_n, G_n')$ \emph{bad} if
    \begin{itemize}
        \item $|S| > \eps n$;
        \item or $|p_r(G_n, H, \tilde{A}) - c_H| > \eps$;
        \item or $| |{\tilde A}| - a n| > \eps n$;
    \end{itemize}

    Then $(G_n, G_n')$ is good whp (with probability $1-o(1)$): the first event does not occur whp by the above argument, the second event does not occur whp by
    Lemma~\ref{lem.probconv}, finally, the third event does not occur whp by the law of large numbers, since $|\tilde A| \sim Binom(n, \mu(A))$.

    We will write as a shorthand $x = b \pm c$ for $x \in [b - c, b+c]$. We write
    $x \pm b = c \pm d$ for $[x-b, x+b] \subseteq [c-d, c+d]$.

    Suppose $(G_n, G_n')$ is good. Then for each vertex $v$ of $G_n$ such that $v \not \in S$ we  have $B_r(G_n,v) \cong B_r(G_n',v)$.
    Write $N(G, Q) = \sum_{v \in Q} \ii_{B_r(G,v) \cong H}$. Then
    \[
        p_r := \frac {N(G_n, \tilde A)} {|{\tilde A}|} = c_H \pm \eps
    \]
    and 
    \[
        N(G_n, {\tilde A}) = a c_H n \left(1 \pm \frac \eps a\right) \left(1 \pm \frac {\eps} c_H\right)  = a c_H n (1 \pm t).
    \]
    Further, since $(G_n, G_n')$ is good, $|N(G_n, {\tilde A}) - N(G_n', A')| \le |S| + |V(G_n') \setminus V(G_n)| \le 2 \eps n$;
    $||{\tilde A}| - |A'|| \le \delta n \le \eps n$ and $||{\tilde A}| - a n| \le \eps n$.
    Thus we have
    \begin{align*}
        & p_r(G_n', H, A') = \frac {N(G_n', A')} {|A'|}  = \frac {N(G_n, {\tilde A}) \pm 2 \eps n} {|\tilde A| \pm  \eps n}
        \\ &= \frac {N(G_n,\tilde A)} {|\tilde{A}|} ( 1 \pm \frac {2 \eps n} {N(G_n, \tilde A)}) (1 \pm  \frac {2\eps n} {|\tilde{A}|})
        \\ &= p_r(G_n, H, {\tilde A}) (1 \pm t)
    \end{align*}
    Here we used simple inequalities $(1 \pm \eps)^{-1} = (1 \pm 2 \eps)$,  $(1 \pm \eps_1) (1 \pm \eps_2) =
    1 \pm 2 (\eps_1 + \eps_2)$ that hold for any $\eps_1, \eps_2, \eps \in (0, 0.5]$ and the facts that
    \[
        \frac {2 \eps n} {N(G_n, \tilde A)} < \frac {4 \eps} {a c_H}  \le 0.5t < 0.5; \qquad \frac {2\eps n} {|\tilde A|} \le \frac {4 \eps} a < 0.5t < 0.5.
    \]
    Thus we have
    \[
        \pr( |p_r(G_n', H, A') - c_H| > t) \le \pr( (G_n, G_n') \mbox{ is bad}) \to 0.
    \]
\end{proof}

Recall that for a random bipartite graph $H'$, $p_r^{(1)}(H, H') = \E (\ii_{B_r(H', u^*) \cong H} | H')$ where $u^* \in_u V_1(H')$.
\begin{lemma} \label{lem.inhomoglocal}
    Suppose $\beta$, $\xi^{(1)}$, $\xi^{(2)}$, $n_1$, $n_2$ and $H_n$ are as in Theorem~\ref{thm.riglocal}(iii).

    Then for any rooted tree $H$,  $p_r^{(1)}(H_n, H) \xrightarrow{p} \bp_r(\mathcal{X}_{\kappa, A_1}, H)$, where $\mathcal{X}_{\kappa}$
    is the branching process corresponding to $(S, \mu)$, $\kappa=\kappa_\beta$ given in (\ref{eq.Sdef}) and (\ref{eq.kappadef}),
    and $\mathcal{X}_{\kappa, A_1}$ is $\mathcal{X}_{\kappa}$ conditioned on the event $A_1$ that the type $X$ of $\rr(\cx_\kappa)$ 
    satisfies $i_X = 1$.
  \end{lemma}

\begin{proof} 
    Let $N = \lfloor (1+\beta) n_1 \rfloor$. 
%
       We will define a coupling between $H_n$ and $H_n' \sim G(N, \kappa)$, as follows.
    Construct sequences of independent random variables
    indexed by the elements of the fixed countable vertex sets $\cv^1, \cv^2$:
    \begin{equation}\label{eq.seqxipre}
        \{ \xi_{v}^{(i)} : v \in \cv^i \}; \quad \xi_{v}^{(i)} \sim \xi^{(i)} \quad v \in \cv^i, \quad  i=1,2,
    \end{equation}
    \[
        \{U_{u,v}, u \in \cv^1, v \in \cv^2\}, \quad U_{u,v} \sim Uniform([0,1]), u \in \cv^1, v \in \cv^2
    \]
    and an independent 
    \begin{equation}\label{eq.seqipre}
        n_1' \sim Binom(N, q_1); \quad n_2' = N - n_1'.
    \end{equation}
    Here, as before,  $q_1 = \frac 1 {1+\beta} > 0$ and $q_2 = 1 - q_1 > 0$.
   
    We construct $H_n'$ as follows. We let $V_i(H_n')$ consist of the first $n_i'$ elements of $\cv^i$.
    For $i =1,2$ and each $u \in V_i(H_n')$ we assign the 
    point $x_u = (i_u, w_u)= (1, {\tilde \xi}_u^{(i)}) \in S$. We define for $u \in \cv^1$, $v \in \cv^2$
    \[
        p_{u,v}' := \frac {\kappa(x_u, x_v)} N =  \frac { {\tilde \xi}_u {\tilde \xi}_v} {2 N},
    \]
    and let $E' = E'(H_n')$ consist of those pairs $uv$, $u \in \cv^1$, $v \in \cv^2$ where 
\[
    U_{u,v} \le p_{u,v}'' := 2 p_{u,v}' - (p_{u,v}')^2.
\]
    Clearly, taking an independent random permutation of the vertices of $H_n'$ gives a copy of $G(N,\kappa)$.

    Similarly, we construct the graph $H_n$ by letting $V_i(H_n)$ consist of the first $n_i$ vertices
    of $\cv^i$, $i = 1,2$, defining for $u \in \cv^1$, $v \in \cv^2$
    \[
        p_{u,v} :=   \frac {\xi^{(1)}_u \xi^{(2)}_v} {\sqrt {n_1 n_2}};
    \]
    and adding $\{u,v\}$ to $E = E(H_n)$ if and only if $U_{u,v} \le p_{u,v}$.

    We let $A = V_1(H_n)$ and $A' = V_1(H_n')$. 
    Denote $\Delta_V = V(H_n) \triangle V(H_n')$ and $\Delta_E = E(H_n) \triangle E(H_n')$.
    Since $n_i' = Binom(N, q_i)$ for $i = 1,2$ we have, using 
    Chebyshev's inequality
    \[
        \E |n_i' - N q_i|  \le n_1^{3/4} + N \pr (|n_i' - N q_i| > n_1^{3/4}) = O\left(n_1^{3/4}\right) 
    \]
    So for $i = 1,2$
    \[
        \E |n_i' - n_i| \le \E |n_i' - N q_i| + |N q_i - n_i| = o(N)
    \]
    and $\E |\Delta_V| = o(N)$. Let $\Delta^1_E$ be 
    the set of edges incident to vertices in $\Delta_V$ combined over both graphs $H_n$, $H_n'$. 
    Let $u_0$ and $v_0$ be the first vertex of $\cv^1$ and $\cv^2$ respectively.
    Using our coupling
    we see that there is a constant $c$ such that for all $n$ large enough
    \[
      \E (|\Delta_E^1| \big | n_1') \le |\Delta_V | N \E (p_{u_0,v_0} + p_{u_0,v_0}'') \le c |\Delta_V|.
    \]
    Here we used the assumption (iii) that $\E \xi^{(i)} < \infty$, $i=1,2$. Hence $\E |\Delta_E^1| = O(\E|\Delta_V|) = o(N)$.

    We have 
    \[
        \frac {\E \left(\xi^{(i)}\right)^2} {N^{1/2}} \le N^{-1/4} \E \xi^{(i)} + N^{-1/2} \E \xi^{(i)} \mathbb{I}_{\xi^{(i)} \ge N^{1/4}}
        = o(n^{-1/2} \E \xi^{(i)}).
    \]
    Thus
    \[
        \E (p_{u,v}'')^2 = o(n_1^{-1} \E \xi^{(1)} \xi^{(2)}).
    \]
    Using also the assumption $n_2 / n_1 \to \beta$ of Theorem~\ref{thm.riglocal} 
    all $u \in \cv^1$, $v \in \cv^2$ 
    \[
        \E |p_{u,v} - p_{u,v}''|  \le {\E \xi^{(1)} \xi^{(2)}}  \left| (n_1 n_2)^{-1/2} - \beta^{-1/2} n_1^{-1} (1+o(1)) \right| = o(n_1^{-1}) 
    \]
    Finally, let $\Delta_E^2 = \Delta_E \setminus \Delta_E^1$, denote by $\E_*$ the conditional
    expectation
    given (\ref{eq.seqxipre}), (\ref{eq.seqipre}), and let $\cv^i_N$ for $i=1,2$ consist of the first $N$ elements of $\cv^i$. We have
    \[
        \E |\Delta_E^2| \le \E \E_* (|\Delta_E^2|) \le  \E \sum_{u \in \cv^1_N} \sum_{v \in \cv^2_N} |p_{j,k}'' - p_{j,k}| = o(N).
    \]
    We have shown that $H_n$ and $H_n'$ are such that
    $\E |\Delta_V| = o(N)$, $\E |\Delta_E| = o(N)$ and $\E |A \triangle A'| = \E |n_1' - n_1| = o(N)$. 
    Now 
    Lemma~\ref{lem.smalldiff} completes the proof.
\end{proof}

\bigskip

\begin{proofof}{Theorem~\ref{thm.riglocal}, case (ii)}
    The claim follows by Lemma~\ref{lem.inhomoglocal} and Proposition~\ref{prop.inhomogbranching}.
\end{proofof}

\subsection{The configuration model}
\label{sec.CM}
    We present a proof based on coupling of the breadth-first exploration of $B_r(H_n, u_1^*)$, $\dots$, 
        $B_r(H_n, u_k^*)$ with the first $r$ generations of $k$ independent branching processes $\ct(D_1, D_2)$, see also \cite{bollobasriordan2011b, gamarnikmisra2014}. We continue with the notation of the beginning of this section, page~\pageref{sec.proofs}.
 
    \bigskip

    \begin{proofof}{Theorem~\ref{thm.riglocal}(iv)} 
       Let $\ct = \ct(D_1, D_2)$. 
       It suffices to show that for arbitrary fixed positive integers $r, k$ and any sequence  $C_1, \dots, C_k$ of rooted trees of radius at most~$r$
       \begin{equation}\label{eq.bfsconvconf}
           \left | \bp_r^{(1)}(H_n, C_1, \dots, C_k) - \prod_{j=1}^k \bp_r(\ct, C_j) \right| \to 0.
       \end{equation}
        Applying this with $k=1$ and $k=2$ and using the same standard second moment argument as in the proof of Theorem~\ref{thm.riglocal}(i) on page~\pageref{pf.riglocal.i-ii}
        shows that $p_r^{(1)}(H_n, C_1) \xrightarrow{p} \bp_r(\ct, C_1)$ which yields (\ref{eq.suff}). 
       In the rest of the proof we show (\ref{eq.bfsconvconf}).

        Let $Y_r$ denote the number of nodes in the first $r$ generations (i.e., generations $0, \dots, r$) of $\ct(D_1, D_2)$. Fix an arbitrary $\eps > 0$ and a positive integer $T$ such that
        $\pr(Y_{r} \ge T/k) < \eps/k$. Such a $T$ exists,
        since $\E D_i < \infty$ and $\pr(D_i = \infty) = \pr(D_i^* = \infty) = 0$ for $i = 1, 2$.

        Recall that for each $n$ the degree sequences $d_1=d_1(n), d_2=d_2(n)$ are indexed by $\Vn 1$, $\Vn 2$ respectively,
        and that in the configuration bipartite graph $H_n$ each vertex $v$ in part $i$ is assigned a list $E_v$ of $d_{i,v}$ half-edges.
        We identify the $j$-th half-edge of $E_v$
        with the pair $(v,j)$ and we think of the set $E_v$ as ordered by 
        the second coordinate of each element.
        There are in total $N = \sum_{u \in \Vn 1} d_{1,u} = \sum_{v \in \Vn 2} d_{2,v}$ half-edges assigned to vertices from each of the parts $\Vn 1$, $\Vn 2$.           
        Let 
        $n_{i,j}$ 
        denote
        the number of elements in $d_i$ that equal~$j$.

        Using the assumption
        of the theorem, we can define $\eps_n \to 0$ 
        such 
        that for all $n=1,2,\dots$, $i = 1,2$ and 
        $j=1,2,\dots$
        \[
            \left|\frac {n_{i,j}} {n_i}  - \pr(D_i = j)\right| \le \eps_n; \quad \left|\frac N {n_i} - \E D_i \right| \le \eps_n.
        \]

        For each $n$ we can construct the following set of $2T+k$ independent random elements:
        \begin{align}
            &\hat{v}_1, \hat{v}_2, \dots, \hat{v}_k \quad \mbox{ uniformly random elements from }\Vn 1;
            \\  &\hat{h}_{11}, \hat{h}_{12}, \dots, \hat{h}_{1T} \mbox{ uniformly random half-edges from } \cup_{v \in \Vn 1} E_v;
            \\  &\hat{h}_{21}, \hat{h}_{22}, \dots, \hat{h}_{2T} \mbox{ uniformly random half-edges from } \cup_{v \in \Vn 2} E_v.
        \end{align}
        Write $\hat{d}_i = d_{1,\hat{v}_i}$, also let $\hat{d}_{i,j} = d_{i,v}$ where $v$ is the end vertex of $\hat{h}_{i,j}$.

        By the assumption (iv) of the theorem we have $\hat{d}_1 \xrightarrow{d} D_1$.
        Now $\hat{d}_{11} \sim \hat{d}_1^*$ and since $\hat{d}_1 \xrightarrow{d} D_1$ and $\E D_1 < \infty$
        by Lemma~\ref{lem.uint} we
        get $\hat{d}_{11} \xrightarrow{d} D_1^*$. Similarly, $\hat{d}_{21} \xrightarrow{d} D_2^*$.
        If a sequence $X_n$ of discrete random variables converges in distribution to a random variable $Y$, then $d_{TV}(X_n, Y) \to 0$.
        For a pair of random variables $(X, Y)$ with $d_{TV}(X,Y) = a$, we can always define a coupled copy $(X',Y')$ such 
        that $X' \sim X$, $Y' \sim Y$ and $\pr(X = Y) \ge 1 - a$, see, e.g., \cite{grimmetstirzaker}. 
        By the union bound and 
        independence
        of $\hat{v}_j$, $\hat{h}_{i,j}$, $j=1,\dots,T$, $i = 1,2$,
        it follows that there is a positive sequence $\delta_n \to 0$ such 
        for each $n$ 
        we can couple $\{\hat{d}_j\}$, $\{\hat{d}_{1,j}\}$ and $\{\hat{d}_{2,j}\}$, $j = 1, \dots, T$
        with random variables
        \begin{align}
            &\tilde{d}_1, \tilde{d}_2, \dots, \tilde{d}_k \sim D_1,                \label{eq.seq1}
            \\  &\tilde{d}_{11}, \tilde{d}_{12}, \dots, \tilde{d}_{1T} \sim D_1^*,     \label{eq.seq2}
            \\  &\tilde{d}_{21}, \tilde{d}_{22}, \dots, \tilde{d}_{2T}  \sim D_2^*,        \label{eq.seq3}
        \end{align}
        so that the variables in $(\ref{eq.seq1}) - (\ref{eq.seq3})$ are 
        independent 
        and
        the event $A$ that $\tilde{d}_j = \hat{d}_j$, for all $j = 1, \dots, k$ and 
        $\tilde{d}_{i,j} = \hat{d}_{i,j}$ for all $i=1,2$, $j  = 1,\dots T$ holds with probability at least $1 - \delta_n$.
       

        We may further extend the sequences (\ref{eq.seq2}) and (\ref{eq.seq3}) by defining independent random variables $\tilde{d}_{i,j} \sim D_i^*$ for each $i=1,2$ and each $j = T+1, T+2,\dots$, independent of the previously defined variables.

        Now use $\{\tilde{d}_1, \dots, \tilde{d}_k\}, \{\tilde{d}_{i,j}, j = 1, 2,\dots\}, i = 1, 2$
        to define the first $r$ generations of  $k$ independent copies of $\ct$ as follows.  Generate the children of each node
        in a breadth-first search manner: let $\tilde{d}_1$ be the number of children of the root of the first tree. Next, if the root has $s$
        children, for the $j$-th child add  $\tilde{d}_{2j}-1$ children. Similarly for nodes in an even generation
        use random variables $\tilde{d}_{1j} - 1$, taking the next unused random variable from the sequence for each new subtree, and so on. Once the degrees of all generations up
        to $r-1$ are determined, start a new rooted tree 
        using the next unused random variable $\tilde{d}_l$ as the degree of its root, and so on. Continue until the degrees of all generations up to $r-1$ of the $k$-th tree are determined and denote the resulting forest of $k$  ordered rooted trees by $\tilde{F}$. 

        In parallel we perform the following BFS-exploration of $H_n$ (truncated at distance $r$) that constructs a forest (an ordered sequence) $F$ of $k$ rooted trees. We will make $T$ steps, at each step updating $F$ (initially empty) and $Q$ (an initially empty sequence, representing the \emph{queue} data structure where new elements are always appended to the end, and the \emph{pop} 
    operation returns the first element of $Q$ and deletes it). To keep the notation simpler we will not index $F$ and $Q$ by $n$ or the step number $j = 1, \dots, T$.
        If after some step $j$ we have $|V(F)| + |Q| \ge T$, we say that step $j$ \emph{overflows} and stop the exploration
        (no more matches are revealed until step $T$). We also stop the exploration if we \emph{complete} it or we \emph{fail} at some step $j$, this is defined below.

        At step 1: let $v = \hat{v}_1$ be the root of the first tree in $F$. Step 1 never fails.
        Add $\hat{d}_1$ half-edges of $\hat{v}_1$ to $Q$ keeping their order. 
        
        Let $y$ be a positive integer, $y < T$. Suppose the steps $1, \dots, y$ did not fail or overflow, and the forest $F$ after step $y$ has $k_y$ trees and $N_y$ vertices.
       We have several possibilities:
       \begin{itemize}

           \item The queue $Q$ is empty and $k_y=k$: the construction is complete. $F$ and $Q$ remain frozen for all steps $y+1, \dots, T$.
               
           \item The queue $Q$ is empty and $k_y < k$. Consider the vertex $v = \hat{v}_{k_{y} + 1}$. If $v$ is
               already in $F$, declare that step $y+1$ fails. Otherwise, start a new tree with root $v$ in $F$ and add the half-edges
               $E_v$ to the end of $Q$.
           \item $Q$ is non-empty. Pop the first half-edge $(u,j')$ from $Q$. 
               Suppose $u \in \Vn i$. Suppose also that we used exactly $s$ half-edges from the sequence $\hat{h}_{\bar{i},1}, \hat{h}_{\bar{i},2} \dots, $ 
               in the previous steps. Pair $(u,j')$ with $\hat{h}_{\bar{i}, s+1} = (v, j'')$. 
               If 
               $v$ is already in $F$,
               we say that step $y+1$ fails. Otherwise, we add $v$ as a child of $u$ in $F$. If
               $v$ is at distance $ < r$ from the root of its tree component, we add the remaining $d_{\bar{i}, v}-1$ half-edges $E_v \setminus \{(v,j'')\}$ to the end of $Q$.
       \end{itemize}

    This process reveals some pairs in the random matching that yields $H_n$. 
       After step $T$ we complete the construction of $H_n$ by picking an independent random matching between
       those half-edges in $\cup_{v \in \Vn 1} E_v$ and $\cup_{v \in \Vn 2} E_v$ which have not yet been paired by the exploration
       process.

       Let $W_j$ be the event that step $j$ fails.
       Let $O_j$ be the event that step $j$ overflows. We can assume that after the first fail or an overflow at step $j$, $W_t$ and $O_t$
       do not occur, and $F$ and $Q$ remain unchanged for all subsequent steps $t = j+1, j+2, \dots, T$. 

       For two ordered sequences of rooted graphs $L = \{L_1, \dots, L_t\}$, $L' = \{L_1', \dots, L_t'\}$ we write $L \cong L'$
       to denote the fact there is a rooted isomorphism that maps $\rr(L_j)$ to $\rr(L_j)$ for $j=1,\dots,t$.

       Let $K$ be the event that no step $j = 1,\dots,T$ fails or overflows and the event $A$ holds. This implies that the process constructing $F$ must have
       finished after some step $j < T$ (otherwise after step $T$ we have $|V(F)| \ge T$).
       $K$ 
       implies that $F$ has no cycles and no repeated edges.
       By our construction on $K$ we have $F \cong  \{B_r(H_n, \hat{v}_1), \dots, B_r(H_n, \hat{v}_k)\}$. (Note that the only unrevealed
       edges of $H$ touching $F$ can be edges incident to generation $r$ vertices in $F$. These edges cannot have both endpoints in $F$ because the (multi-)graph $H$ is bipartite.)

       Since the BFS procedure builds the forests $F$ and $\tilde{F}$ in the same order and the degrees of nodes are coupled to agree, the event
       $K$ implies that $\tilde{F} \cong F$.

       Now if $A$ holds and $F$ overflows at some step $j$ then our coupling implies that the final forest has $|V(\tilde{F})| \ge T$.
       So
       \[
           \pr(\tilde{F} \cong F) \ge  1 - \pr(\bar{A}) - \pr(|V(\tilde{F})| \ge T) - \pr(\cup_j W_j) .
       \]
       We have shown above that $\pr(\bar{A}) \to 0$. By our choice of $T$ and the union bound, $\pr(|V(\tilde{F})| \ge T) \le k \eps k^{-1} = \eps$.
       Now let $\cf_j$ be the $\sigma$-algebra generated by 
       the first $j$ steps (i.e., the random variables 
       $\{\hat{v}_l\}$, 
       $\{\hat{h}_{1l}\}$ and 
       $\{\hat{h}_{2l}\}$ that are revealed in steps $l=1, \dots, j$). 
       Let us bound $\pr(W_{j+1} | \cf_j) := \E (\ii_{W_{j+1}} | \cf_j)$.
       On the blocks of $\cf_j$ where $W_l$ or $O_l$ occurs for some $l \le j$ we have that $\ii_{W_{j+1}} = 0$ by definition. We
       may consider those blocks of $\cf_j$ where $W_1, \dots, W_j$ and $O_1, \dots, O_j$ do not occur.
       On those blocks of $\cf_j$ where at step $j+1$ we add a new root vertex, $W_{j+1}$ occurs with probability
       at most
       \[
           \frac {N_j} {n_1} \le \frac {T} {n_1}. 
       \]
       On those blocks of $\cf$ where at step $j+1$ we consider a new half-edge from $Q$, $W_{j+1}$ occurs with probability
       at most $2 T / N$, since the total number of half-edges 
       assigned to vertices in $F$ is exactly $2|E(F)| + |Q| \le 2 T$.

       Thus there
       are absolute constants $N_0, c$  for all $n \ge N_0$
       such that
       \[
           \pr(W_{j+1} | \cf_j) \le \frac {T} {n_1} + \frac {2 T} {\E D_1 n_1 (1-\eps_n)} \le  \frac {c T} {n_1}.
       \]
       By the union bound
       \[
           \pr(W_1 \cup \dots \cup W_T) \le \sum_{j=0}^{T-1} \sup \pr (W_{j+1} | \cf_j) \le \frac {c T^2} {n_1} \to 0.
       \]
       Therefore $\pr(\tilde{F} \cong F)  \ge 1 - \eps - o(1).$
       Since the proof holds for arbitrary $\eps > 0$, we conclude that $\pr(\tilde{F} \cong F) \to 1$ as $n \to \infty$.
       This completes the proof of (\ref{eq.bfsconvconf}).  

   \end{proofof}
  
   \bigskip

   \begin{proofof}{Remark~\ref{rmk.remark1}} 
       The remark is shown as part of the proof of each of the cases, see Lemma~\ref{prop.nocorrelationactive}, Lemma~\ref{prop.inhomogbranching}, (\ref{eq.seq1})--(\ref{eq.seq2}) and the argument in the proof.
   \end{proofof}

   \bigskip

   \begin{proofof}{Remark~\ref{rmk.remark2}} \hskip 1mm  
       Let $n_1 \to \infty$ be a sequence of positive integers.
       Let $\beta = \E D_1' (\E D_2')^{-1}$ and $n_2 = \lfloor \beta n_1 \rfloor$.
       Let $d_i' = d_i'(n)$ consist of $n_i$ independent copies of $D_i'$, so that $d_1'(n)$, $d_2'(n)$ are independent. 
       Write $S'_i = \sum_j d_i'$.  To make the sums of both sequences equal, define $Z_i = (S_i' - S'_{\bar i})_+$ and let $d_1 = d_1(n) = \{d_1', \dots, d_{n_i}', Z_i\}$. 

       Fix any $\eps > 0$. 
       By our choice of $n_2$, $|\E S_2' - \E S_1'| = o(n_1)$.
       By the weak law of large numbers  $\frac{S'_i} {n_i} \xrightarrow{p} \E D_i'$, therefore for all $n$ large enough
       \[
           \pr ( |S'_1 - S'_2| > 3 \eps n_1) \le \pr ( |S'_1 - \E S_1'| > \eps n_1) + \pr ( |S'_2 - \E S_2'| > \eps n_1) + o(1) \to 0.
       \]
       So $Z_i (n_i +1)^{-1} \le |S'_1 - S'_2| (n_i +1)^{-1}  \xrightarrow{p} 0$ and 
       \[
           \frac {\sum_{j=1}^{n_i+1} d_{i,j}} {n_i + 1} = \frac {n_i} {n_i+1} \frac{S_i'} {n_i} + \frac {Z_i} {n_i +1 } \xrightarrow{p} \E D_i'.
       \]
       Similarly for $i = 1,2$ and any $k=0, 1,\dots$:
       \[
           \frac {\sum_{j=1}^{n_i+1} \ii_{d_{i,j} = k}} {n_i + 1} = \frac {n_i + 1} {n_i}  \frac {\sum_{j=1}^{n_i} \ii_{d_{i,j}'=k}} {n_i} + \frac {\ii_{Z_j = k}} {n_i+1} \xrightarrow{p} \pr(D_i' = k).
       \]
   \end{proofof}


\begin{thebibliography}{45}
    \bibitem{aldouslyons} D. Aldous and R. Lyons. Processes on unimodular random networks. \emph{Electron. J. Probab.} \textbf{12} (2007), no. 54, 1454--1508.

    \bibitem{aldoussteele2004} D. Aldous and J. M. Steele, The objective method: probabilistic combinatorial optimization and local weak convergence, in \emph{Probability on discrete structures, Encyclopaedia Math. Sci.} \textbf{110} (2004), Springer, Berlin, 1--72.


    \bibitem{andersson1998} H. Andersson, Limit Theorems for a Random Graph Epidemic Model, \emph{Ann. Appl. Probab.} \textbf{8} (1998) 1331--1349.

    \bibitem{arratiagoldstein2009} R. Arratia, L. Goldstein and F. Kochman, Size bias for one and all, \emph{Probab. Surv.} \textbf{16} (2019) 1--61.

    \bibitem{benjaminischramm2001} I. Benjamini and O. Schramm, Recurrence of distributional limits of finite planar graphs, \textit{Electron. J. Probab.} \textbf{6} (2001) no. 23 1--13.

    \bibitem{benjaminilyonsschramm2013} I. Benjamini, R. Lyons and O. Schramm, Unimodular random trees, \textit{Ergodic Theory Dyn. Syst.} \textbf{35} (2013) 1--15.

    \bibitem{bergerborgschayessaberi2014} N. Berger, C. Borgs, J. T. Chayes and A. Saberi, Asymptotic behavior and distributional limits of preferential attachment graphs, \emph{Ann. Probab.} \textbf{42} (2014) 1--40.


    \bibitem{billingsley_1999} P. Billingsley, \emph{Convergence of probability measures}, John Wiley \& Sons, New York, Second Edition (1999).

    \bibitem{billingsley_weak} P. Billingsley, \emph{Weak convergence of measures: applications in probability}, Society for industrial and applied mathematics, Philadelphia (1971).   

    \bibitem{bloznelis2008} M. Bloznelis, Degree distribution of a typical vertex in a general random intersection graph, \textit{Lith. Math. J.} \textbf{48} (2008) 38--45.
    
    \bibitem{bloznelis2010} M. Bloznelis, The largest component in an inhomogeneous random intersection graph with clustering, \textit{Electron. J. Combin.} (2010) \textbf{17} \#R110.

    \bibitem{bloznelis2011} M. Bloznelis, Degree and clustering coefficient in sparse random intersection graphs, \textit{Ann. Appl. Probab.} \textbf{23} (2013) 1254--1289. 

    \bibitem{bloznelis2014} M. Bloznelis, Degree-degree distribution in a power law random intersection graph with clustering, \emph{In D. F. Gleich, J. Komjathy, N. Litvak (Eds.): Algorithms and Models for the Web Graph -- 12th International Workshop, WAW 2015}, Lecture Notes in Computer Science \textbf{9479} (2015) 42--53.
        
    \bibitem{bloznelisdamarackas2013} M. Bloznelis and J. Damarackas, Degree distribution of an inhomogeneous random intersection graph, \emph{Electron. J. Combin.} \textbf{20} (2013) \#P3.


    \bibitem{bgjkr2014a} M. Bloznelis, J. Jaworski, E. Godehardt, V. Kurauskas, and K. Rybarczyk,  Recent progress in complex network analysis -- models of random intersection graphs, \emph{In B. Lausen et al. (eds) Data Science, Learning by Latent Structures, and Knowledge Discovery. Studies in Classification, Data Analysis, and Knowledge Organization}, Springer-Verlag Berlin Heidelberg (2015) 69--78. 

    \bibitem{bgjkr2014b} M. Bloznelis, J. Jaworski, E. Godehardt, V. Kurauskas, and K. Rybarczyk,  Recent progress in complex network analysis -- properties of random intersection graphs, \emph{In B. Lausen et al. (eds) Data Science, Learning by Latent Structures, and Knowledge Discovery. Studies in Classification, Data Analysis, and Knowledge Organization}, Springer-Verlag Berlin Heidelberg (2015) 79--88. 

    \bibitem{bloznelis2013} M. Bloznelis, J. Jaworski and V. Kurauskas, Assortativity and clustering in sparse random intersection graphs, \emph{Electron. J. Probab.} \textbf{18} (2013) no. 38 1--24.
       
    \bibitem{blozneliskurauskas2013} M. Bloznelis and V. Kurauskas, Large cliques in sparse random intersection graphs, \emph{Electron. J. Combin.} \textbf{24} (2017) \#P2.5.

    \bibitem{bollobasjansonriordan2007} B. Bollob\'as, S. Janson and O. Riordan, The phase transition in inhomogeneous random graphs. \emph{Random Struct. Algor.} \textbf{31} (2007) 3--122.

    \bibitem{bollobasjansonriordan2011} B. Bollob\'as, S. Janson, and O. Riordan, Sparse random graphs with clustering, \emph{Random Struct. Algor.} \textbf{38} (2011) 269--323.

    \bibitem{bollobasriordan2011} B. Bollob\'as and O. Riordan, Sparse graphs: metrics and random models, \emph{Random Struct. Algor.}, \textbf{39} (2011) 1--38.

    \bibitem{bollobasriordan2011b} B. Bollob\'as and O. Riordan, An old approach to the giant component problem,
        \emph{J. Comb. Theory B} \textbf{113} (2015) 236--260.

    \bibitem{bordenavelelarge2009} C. Bordenave and M. Lelarge, Resolvent of large random graphs, \emph{Random Struct. Algor.}, \textbf{37} (2010) 332--352.

    \bibitem{csikvarilin2013} P. Csikv\'{a}ri and L. Zhicong, Graph homomorphisms between trees, \textit{Electron. J. Probab.} \textbf{21} (2014) P4.9.

    \bibitem{dembomontanari2010} A. Dembo and A. Montanari, Gibbs measures and phase transitions on sparse random graphs, \emph{Braz. J. Probab. Stat.} \textbf{24} (2010) 137--211.


    \bibitem{durrett} R. Durrett, \emph{Probability: theory and examples}, Cambridge University Press, Fourth Edition (2010).

    \bibitem{fiolgarriga2012} M. A. Fiol and E. Garriga, Number of walks and degree powers in a graph, \emph{Discrete Math.} {\textbf 309} (2009) 2613--2614.

    \bibitem{gamarnikmisra2014} D. Gamarnik and S. Misra, Giant component in multipartite graphs with given degree sequences, \emph{Stoch. Syst.} \textbf{5} (2015) 372--408.

    \bibitem{georgakopauloswagner2015} A. Georgakopoulos and S. Wagner, Limits of subcritical random graphs and random graphs with excluded minors, \href{https://arxiv.org/abs/1512.03572}{arXiv:1512.03572} (2015).

    \bibitem{gjr2012} E. Godehardt, J. Jaworski, and K. Rybarczyk, Clustering coefficients of random intersection
        graphs, \emph{in: Challenges at the Interface of Data Analysis, Computer Science, and Optimization}, 
        Springer, Berlin, 2012. 243--253. 

    \bibitem{grimmetstirzaker} G. Grimmett and D. Stirzaker, \emph{Probability and random processes}, Oxford University Press (2001).

    \bibitem{guillaumelatapy2006} J.-L.~Guillaume and M. Latapy, Bipartite graphs as models of complex networks, \emph{Physica A: Statistical Mechanics and its Applications} \textbf{371} (2006) 795--813.

    \bibitem{vanderhofstad2020} R. van der Hofstad, \emph{Random Graphs and Complex Networks}, Volume II, preliminary version, 
        \url{https://www.win.tue.nl/~rhofstad/NotesRGCNII_colleagues_25_04_2022.pdf}, accessed 2022-06-09.

    \bibitem{jansonluczak2009} S. Janson, and M. J. Luczak, A new approach to the giant component problem, \emph{Random Struct. Algor.}, \textbf{34} (2009) 197--216.

    \bibitem{litvakhofstad2013} N. Litvak and R. van der Hofstad. Uncovering disassortativity in large scale-free networks, \emph{Phys. Rev. E} \textbf{87} (2013) 022801.

    \bibitem{lovasz-hom-book} L. Lov\'{a}sz, \emph{Large networks and graph limits}, Colloqium Publications, American Mathematical Soc. (2012).

    \bibitem{lyons} R. Lyons, Asymptotic enumeration of spanning trees, \emph{Combin. Probab. Comput.} \textbf{14} (2005) 491--522.

  \bibitem{kssc99}  M. Karo{\'n}ski, E. R. Scheinerman and K. B. Singer-Cohen, On Random Intersection Graphs: The Subgraph Problem, \textit{Combin. Probab. Comput.} \textbf{8} (1999) 131--159.

    \bibitem{karrernewman2010} B. Karrer and M. E. J. Newman, Random graphs containing arbitrary distributions of subgraphs, \emph{Phys. Rev. E} \textbf{82} (2010) 066118.

    \bibitem{kurauskas2015} V. Kurauskas, On local weak limit and subgraph counts for sparse random graphs, \href{https://arxiv.org/abs/1504.08103v2}{arXiv:1504.08103v2} (2015).

    \bibitem{newmanstrogatzwatts2001} M. E. J. Newmann, S.H. Strogatz and D.J. Watts, Random graphs with arbitrary
        degree distributions and their applications, \emph{Phys. Rev. E} \textbf{64} (2001) 026118.

    \bibitem{psw2016} K. Panagiotou, B. Stufler and K. Weller, Scaling limits of random graphs from subcritical classes, \emph{Ann. Prob.} \textbf{44} (2016) 3291--3334.


    \bibitem{richardsonurbanke2008} T. Richardson and R. Urbanke, \emph{Modern Coding Theory}, Cambridge University Press (2008).

    \bibitem{salez2011} J. Salez, \emph{Some implications of local weak convergence for sparse random graphs}, PhD dissertation (2011).

    \bibitem{sidorenko1994} A. Sidorenko, A partially ordered set of functionals corresponding to graphs, \emph{Discrete Math.}, {\textbf 131} (1994) 263--277.

    \bibitem{stufler2019} B. Stufler, Local convergence of random planar graphs, \emph{In Ne\v{s}etril et al. (eds), Extended Abstracts EuroComb 2021}, Trends in Mathematics vol. 14 (2021), Birkhäuser, Cham, p. 57--63.

    \bibitem{vadon2019}  V. Vadon, J. Komjáthy, R. and van der Hofstad, A new model for overlapping communities with arbitrary internal structure, \emph{Applied Network Science} \textbf{4} (2019) no. 42 1--19.


    \bibitem{wormald1999} N. C. Wormald. Models of random regular graphs, in \emph{Surveys in Combinatorics}, 1999 (Canterbury), vol 267 of London Math. Soc. Lecture Note Ser., pages 239--298. Cambridge Univ. Press, Cambridge (1999).
    \end{thebibliography}
\end{document}